\documentclass[reqno]{amsart}
\usepackage{amsmath,amssymb,dsfont}
\usepackage{graphics,epsfig,color}
\usepackage[mathscr]{euscript}
\theoremstyle{plain}
\newtheorem{theorem}{Theorem}[section]
\newtheorem{proposition}[theorem]{Proposition}
\newtheorem{lemma}[theorem]{Lemma}

\newtheorem{condition}[theorem]{Condition}
\newtheorem{remark}[theorem]{Remark}

\numberwithin{equation}{section}
\numberwithin{theorem}{section}
\newcommand{\mc}[1]{{\mathcal #1}}
\newcommand{\bb}[1]{{\mathbb #1}}

\newcommand{\id}{{1 \mskip -5mu {\rm I}}}
\renewcommand{\epsilon}{\varepsilon}
\renewcommand{\tilde}{\widetilde}
\renewcommand{\hat}{\widehat}
\newcommand{\supp}{\mathop{\rm supp}\nolimits}
\renewcommand{\div}{\mathop{\rm div}\nolimits}
\newcommand{\Ent}{\mathop{\rm Ent}\nolimits}

\newcommand{\sce}{\mathop{\rm sc^-\!}\nolimits}

\definecolor{light}{gray}{.9}

\title[LDP of the empirical flow]{Large deviations of the empirical
  flow for continuous time Markov chains}

\author[L.\ Bertini]{Lorenzo Bertini}
\address{Lorenzo Bertini \hfill\break \indent
   Dipartimento di Matematica, Universit\`a di Roma `La Sapienza'
   \hfill\break \indent
   P.le Aldo Moro 2, 00185 Roma, Italy}
 \email{bertini@mat.uniroma1.it}

\author[A.\ Faggionato]{Alessandra Faggionato}
\address{Alessandra Faggionato \hfill\break \indent
  Dipartimento di Matematica, Universit\`a di Roma `La Sapienza'
  \hfill\break \indent
  P.le Aldo Moro 2, 00185 Roma, Italy}
\email{faggiona@mat.uniroma1.it}

\author[D.\ Gabrielli]{Davide Gabrielli}
\address{Davide Gabrielli \hfill\break \indent
  Dipartimento di Matematica, Universit\`a dell'Aquila
  \hfill\break\indent
  Via Vetoio,   67100 Coppito, L'Aquila, Italy}
\email{gabriell@univaq.it}

\begin{document}

\begin{abstract}
  We consider a continuous time Markov chain on a countable state
  space and prove a joint large deviation principle for the empirical
  measure and the empirical flow, which accounts for the total number
  of jumps between pairs of states. We give a direct proof using
  tilting and an indirect one by contraction from the empirical
  process.

\bigskip

\noindent \textsc{R\'{e}sum\'{e}}. \hspace{0.05cm}
On consid\`ere une cha\^ine de Markov en temps continu \`a espace d'\'etats
den\'ombrable, et on prouve un principe de grandes d\'eviations commune
pour la mesure empirique et le courant empirique, qui repr\'esente le
nombre total de sauts entre les paires d'\'etats. On donne une preuve
directe \`a l'aide d'un \emph{tilting}, et une preuve indirecte par
contraction, \`a partir du processus empirique.

\bigskip


\medskip

\end{abstract}

\maketitle
\thispagestyle{empty}

\section{Introduction}
\label{introduzione}

One of the most important contribution in the theory of large
deviations is the series of papers of Donsker and Varadhan
\cite{DV4} where the authors develop a general approach to the study
of large deviations for Markov processes both in continuous and
discrete time.  They establish large deviations principles (LDP) for
the empirical measure and for the empirical process associated to a
Markov process.  Given a sample path of the process on the finite time
window $[0,T]$, the corresponding empirical measure is a probability
measure on the state space that associates to any measurable subset
the fraction of time spent on it. A LDP for the empirical measure is
usually called a \emph{level 2} LDP.  Given a sample path, the
corresponding empirical process is a probability measure on paths
defined on the infinite time window $(-\infty,+\infty)$.  More
precisely, it is the unique stationary (with respect to time shift)
probability measure that gives weight $1$ to $T$--periodic paths
such that there exists a period $[t,t+T]$ where they coincide
with the original sample path. A LDP for the empirical process is
usually called a \emph{level 3} LDP.


The large deviations asymptotic of discrete time Markov chains on a
countable state space can be described as follows,
see for example \cite{DZ,dH}.
The rate function for the level $3$ LDP is the relative entropy per
unit of time. The rate function for the level $2$ LDP has instead in
general only a variational representation, which cannot be solved
explicitly even for reversible transition probabilities.  A very
natural and much studied object is the $k$--symbols empirical
measure. This is a probability measure on strings of symbols with
length $k$ obtained from the frequency of appearance in the sample
path.  With a suitable periodization procedure the $k$--symbols
empirical measures constitute a consistent family of measures that are
exactly the $k$ marginals of the empirical process. For each $k>1$,
and in particular for $k=2$ the rate function for the LDP associated
to the $k$ symbols empirical measure has an explicit expression.

The aim of this paper is to provide an analogous picture for
continuous time Markov chains on a countable state space.
For the empirical process the rate function is always the relative entropy
per unit of time. For the empirical measure the rate function has instead only
a variational representation.  In the case of reversible Markov chains
the corresponding variational problem can be solved and the rate
function is related to the Dirichlet form.
In the continuous time setting the natural counterpart  of the
$2$--symbols empirical measure is the empirical flow that can be
defined as follows.
Given a sample path of the Markov chain in the finite time window
$[0,T]$, the corresponding empirical flow is the positive measure
on the pairs of states assigning to each pair a weight given by
the corresponding number of jumps
per unit of time.

As in the discrete time setting, the joint rate function for the
empirical measure and flow can always be written in a closed form (formula
\eqref{rfq} below).
This joint rate function for the empirical measure and flow first
appeared in applied contexts. Originally in information technology
\cite{dlF,WK} and more recently in statistical mechanics
\cite{maes}. In particular, in \cite{WK} it has been used to recover
by contraction the Donsker--Varadhan rate function for the empirical
measure in the case of a state space with only two elements. Being a
LDP intermediate among level 2 and level 3, the authors called it a
level 2.5 LDP.  Later in \cite{BP}, motivated by statistical
applications, the authors have showed that the contraction on the
empirical measure of the rate function proposed by \cite{WK} leads to
the Donsker-Varadhan rate function in the case of finite state space.
In \cite{dlF} a weak level 2.5 LDP has been proved. Finally in
\cite{maes} LDPs for flows and currents have been discussed in
relation to non equilibrium thermodynamics.

In the present paper we give a rigorous proof of a full LDP for Markov
chains on a countable state space.  In the case of infinite state
space, the empirical flow exhibits novel phenomena with respect to the
empirical measure.  In particular,
the Markov chain
could perform  very long excursions towards infinity  in very short time. Therefore, the exponential tightness of the empirical flow
requires additional conditions and poses non trivial topological
issues. We have solved these problems by introducing the bounded weak*
topology and adding an extra condition with respect to the Donsker-Varadhan
ones (see Item (vi) in Condition 2.2). This condition is sharp
as it is also necessary for the exponential tightness of
the empirical flow in the case of  birth and dead processes.
Another technical issue is whether the LDP for the empirical measure
and flow holds in a stronger topology. For the empirical flow a
natural candidate is the strong $L^1$ topology. However, as shown in
the case of birth and dead processes,
the rate function has not in general compact level sets in
the strong $L^1$ topology for flows.

We present two different proofs of the LDP of the empirical measure
and flow. A direct derivation is obtained using a perturbation of the
original Markov measure (under the additional assumption that the
graph underlying the Markov chain is locally finite), while an
indirect derivation is obtained by contraction from the level 3 LDP.
In the last case, the contraction principle (which anyway requires
some work as the map is not continuous and the topology is not
metrizable) leads to a partial result.  In fact, the main point is the
identification of the rate function obtained by contraction with the
closed form \eqref{rfq}. In order to prove this identity - which does
require additional conditions with respect to the level 3 LDP - we
need a geometrical analysis of divergence--free flows on graphs
(see Section~\ref{s:geometria}) which plays a fundamental role also in
direct proof of the LDP by exponential tilting.
Besides the Donsker-Varadhan conditions, level 2 and 3 LDPs can be
proven under hypercontractivity condition, see \cite{DeS}.
Also in this framework, the exponential tightness of the empirical flow
requires the additional condition that the inverse of the mean
holding time  has a finite exponential
moment with respect to the invariant measure.

We mention some recent results about fluctuations of currents
and flows inspiring and motivating the present work. We already
mentioned the paper \cite{maes}.
In \cite{KKT,K} LDPs for the current  of the
Brownian motion on a compact Riemann manifold are obtained.  We
mention also the recent preprint \cite{MSZ} on the joint large
deviations for the empirical measure and flow for a renewal process on
a finite graph. Currents play also a crucial role in biochemical
processes, and the study of large fluctuations and related symmetries
have recently received much attention (see e.g. \cite{FD,LM}
and references therein).
As development of the result given here,
in \cite{BFG1} we recover the LDP for the empirical measure by
contraction from the joint LDP proved here.
In \cite{BFG2} we shall discuss several applications and consequences
of our results like LDPs for currents and  Gallavotti--Cohen
symmetries. In \cite{BFG2}
we will also give sufficient
conditions leading to the joint LDP for empirical measure and flow
when endowing the flow space of the strong $L^1$ topology. We also mention that in \cite{BaBe}
the scheme proposed here  has been extended to the case of continuous time jump processes with an absorbing state,  motivated by the  study of energy transport in insulators.

We finally outline some possible applications of the LDP for empirical
measure and flow in the context of interacting particle systems.
(i) The LDP for the total number of jumps per unit of time has been
recently analyzed in \cite{BLT,BT} for some constrained interacting
particle systems including the east model. In the limit of infinitely
many particles, the associated rate function exhibits a non--trivial
zero level set thus leading to second order large deviations. The
second order rate function is conjectured and partially proven in
\cite{BT}.  This problem can be attacked, by a purely variational
procedure, starting from the joint rate function for empirical measure
and flow.
(ii) In the context of hydrodynamic scaling limits the LDP of
the current has been analyzed in
\cite{BDGJL1,BDGJL2,bd,bd1}. In this
setting a natural problem is the large deviation properties of the
time averaged hydrodynamical current in the large time limit. The
corresponding rate function exhibits interesting phenomena.  On the
other hand, one can take the large time limit before the limit of
infinitely many particles. As the hydrodynamical current can be
written in terms of the empirical flow one can take the scaling limit
in the joint LDP for the empirical measure and flows. If all goes well
one then recovers the hydrodynamical rate function.  In the special
case of the one--dimensional boundary driven zero range process, the
LDP for the current of particles across an edge of the lattice has
been computed by combinatorial techniques in \cite{HRS} based on a
suitable ansatz. In the limit of infinitely many particles it yields
the hydrodynamical result. In principle, this problem, including the
validity of the ansatz in \cite{HRS}, could be addressed starting from
the joint LDP for the empirical measure and flow.
(iii) Always in the context of hydrodynamical scaling limit, the LDP
for the net flow of particles across a segment of the two-dimensional
torus has been analyzed in \cite{bdl}. In particular, it is shown that
the large deviations asymptotic degenerates due to the occurrence of
small vortices near the endpoints of the segment. A non-trivial LDP
should hold in a suitable logarithmic rescaling. This phenomenon can
be analyzed already for a single random walk for which it becomes a
problem on the scaling limit of the rate function here derived.


\section{Notation and results}
\label{definizioni}

We consider a continuous time Markov chain $\xi_t$, $t \in \bb R_+$
on a countable (finite or infinite) state space $V$. The Markov
chain is defined in terms of the \emph{jump rates} $r(x,y)$, $x \not
=y$ in $V$, from which one derives the holding times and the jump
chain \cite[Section 2.6]{N}. Since the holding time at $x\in V$ is
an exponential random variable of parameter
 $r(x): =\sum_{y\in V}
r(x,y)$,  we need to assume  that $r(x)<+\infty$ for any $x \in V$.

The basic assumptions on the chain are the following:
\begin{itemize}
\item[(A1)]
  for each $x \in V$, $r(x)= \sum _{y \in V} r(x,y)$ is finite and
  strictly positive;
\item[(A2)]
   for each $x \in V$ the Markov chain $\xi^x_t$ starting from $x$
   has no explosion a.s.;
\item[(A3)]
  the Markov chain is irreducible, i.e.\ for each $x,y \in V$ and $t>0$
  the event $\{\xi^x_t=y\}$ has strictly positive probability;
\item[(A4)]
  there exists a unique invariant probability measure, that is denoted
  by $\pi$.
\end{itemize}

As in \cite{N}, by invariant probability measure $\pi$ we mean a
probability measure on $V$ such that
\begin{equation}
  \label{invariante}
  \sum _{y \in V} \pi(x) \, r(x,y)
  = \sum _{y \in V} \pi(y) \, r(y,x)\qquad \forall \:x \in V
\end{equation}
where we understand $r(x,x)=0$. We recall some basic facts from
\cite{N}, see in particular Section~3.5 and Theorem  3.8.1 there.
Assuming (A1) and irreducibility (A3), assumptions (A2) and (A4)
together are equivalent to the fact that all states are positive
recurrent.
  In (A4) one could remove the assumption of
uniqueness of the invariant probability measure, since for an
irreducible Markov chain there can be at most only one.
 Under the above assumptions,
$\pi(x)>0$ for all $x \in V$, the Markov chain starting with
distribution $\pi$ is stationary (i.e.\ its law  is left invariant
by time-translations), and the ergodic theorem holds, i.e.\ for any
bounded function $f: V \to \bb R$ and any initial distribution
\begin{equation}
  \label{ergodico}
  \lim_{T \to +\infty} \frac{1}{T} \int_0^T\!dt\, f(\xi_t)
  = \langle\pi, f\rangle
  \qquad \textrm{a.s.}
\end{equation}
where $\langle\pi,f\rangle$ denotes the expectation of $f$ with
respect to $\pi$. Finally, we observe that if $V$ is finite then
(A1) and (A2) are automatically satisfied, while (A3) implies (A4).

We consider $V$ endowed with the discrete topology and the
associated Borel $\sigma$-algebra given by the collection of all the
subsets of $V$.  Given $x\in V$, the distribution of the Markov
chain $\xi^x_t$ starting from $x$, is a probability measure on the
Skorohod space $D(\bb R_+;V)$ that we denote by $\bb P_x$. The
expectation with respect to $\bb P_x$ is denoted by $\bb E_x$.
In the sequel we consider $D(\bb R_+;V)$ equipped with
the Skorohod topology, the associated Borel $\sigma$--algebra, and
the canonical filtration.
The canonical coordinate in $D(\bb R_+;V)$ is denoted by $X_t$.
The set of probability measures on $V$ is denoted by $\mc P(V)$ and it
is considered endowed with the topology of weak convergence and the
associated Borel $\sigma$-algebra. Since $V$ has the discrete
topology, the weak convergence of $\mu_n$ to $\mu$ in $\mc P (V)$ is
equivalent to the pointwise convergence of $\mu_n(x)$ to $\mu(x)$ for
any $x\in V$.

\subsection{Empirical measure and empirical flow}
\label{s:emef}

Given $T>0$ the \emph{empirical measure}
$\mu_T\colon D(\bb R_+;V)\to \mc P(V)$ is defined by
\begin{equation*}
  \mu_T \, (X) = \frac 1T\int_0^T\!dt \, \delta_{X_t}
\end{equation*}
where $\delta_y$ denotes the pointmass at $y$.
Given $x\in V$, the ergodic theorem \eqref{ergodico} implies that the
empirical measure $\mu_T$ converges $\bb P_x$ a.s.\ to $\pi$ as $T \to
\infty$.
In particular, the sequence of probabilities $\{\bb P_x \circ
\mu_T^{-1}\}_{T>0}$ on $\mc P(V)$ converges to $\delta_\pi$.


We denote 
by $E$ the countable set of ordered edges in $V$
with strictly positive jump rate:
 \begin{align*}
& E:= \{(y,z) \in V\times V  \,:\, r(y,z)>0\}\,.
\end{align*}
For each $T>0$ we  define  the \emph{empirical flow} as the map $Q_T
\colon D(\bb R_+;V)\to [0,+\infty]^E $ given by  
\begin{equation}
  \label{montecarlo}
  Q_T(X)
  :=
  \frac{1}{T} \sum_{    t\in [0,T]\,:\:  X_{t^-} \neq X_{t^{\phantom{-}}}}
  \delta_{(X_{t^-} ,X_{t}^{\phantom{-}})}
\end{equation}
Namely, $ T Q_T (X) \, (y,z)$ is the  number of jumps from
$y$ to $z$ in the time interval $[0,T]$ of the path $X$.

\begin{remark}\label{SD}
By the graphical construction of the Markov chain, the random fi\-eld
$\{T Q_T (y,z) \} _{(y,z) \in E}$ under $\bb P_x$ is stochastically
dominated by the random field $\{\mc Z_{y,z}\}_{(y,z) \in E}$
given by independent Poisson  random variables,  $\mc Z_{y,z}$
having mean $Tr(y,z)$.
\end{remark}

We denote by $L^1(E)$ the collection of absolutely summable
functions on $E$ and by $\| \cdot \|$ the associated $L^1$--norm.
The set of nonnegative elements of $L^1(E)$ is denoted by
$L^1_+(E)$.
Since the chain is not explosive, for each
$T>0$ we 
have $\bb P_x$ a.s. that $Q_T\in L^1_+(E)$.

Given a flow $Q\in L^1_+ (E) $ we let its \emph{divergence} $\div Q
\colon V\to \bb R$ be the function defined by
\begin{equation}
  \label{divergenza_fluss}
  \div Q \, (y)= \sum _{z : \, (y,z)\in E} Q(y,z)- \sum_{z:\, (z,y)\in E} Q(z,y),
  \qquad y\in V.
\end{equation}
Namely, the divergence of the flow $Q$ at $y$ is given by the
difference between the flow exiting from $y$ and the flow entering
into $y$. Observe that the divergence maps $L^1_+(E)$ to $L^1(V)$.


Finally, to  each probability $\mu \in \mc P(V)$
we associate the flow $Q^\mu\in \bb R_+^E$ defined by
\begin{equation}
  \label{Qmu}
  Q^\mu(y,z) := \mu(y) \, r(y,z)
  \qquad (y,z)\in E.
\end{equation}
Note that  $Q^\mu\in L^1_+(E) $ if and only if $\langle\mu,
r\rangle<+\infty$. Moreover, in this case, by \eqref{invariante}
$Q^\mu$ has vanishing divergence if only if $\mu$ is invariant for
the Markov chain $\xi$, i.e. $\mu=\pi$.

We now discuss the law of large numbers for the empirical flow.
 As
follows from simple computations (see  \cite[Lemma\,II.2.3]{S} and
\cite[App.\,1, Lemma\,5.1]{KL}, which have to be  generalized  to
the case of unbounded $r(\cdot)$ by means of  \cite[Sec.\,2.8]{N}
and Remark \ref{SD}) for each $(y,z)\in E$ the process
\begin{equation}\label{dacitare}
  M_T \, (y,z) := T \, Q_T (X)\, (y,z) - \int_0^T\!dt \, \delta_y(X_t) r(y,z)
\end{equation}
is a martingale with respect to $\bb P_x$, $x\in V$. Moreover, the
predictable quadratic variation of $M_T(y,z)$, denoted by $\langle
M(y,z)\rangle_T$ is given by
\begin{equation*}
  \langle M(y,z)\rangle_T = \int_0^T\!dt \, \delta_y(X_t) r(y,z)\, .
\end{equation*}
In view of the ergodic theorem \eqref{ergodico}, we conclude that
for each $x\in V$ and $(y,z)\in E$ the family  of real random
variables $Q_T(y,z)$ converges, in probability with respect to $\bb
P_x$, as $T\to+\infty$ to $Q^\pi(y,z)$. We refer to Remark
\ref{willy} for an alternative proof.

\subsection{Compactness conditions}
The classical Donsker-Varadhan theorem \cite{DZ,DeS,DV4,Vld}
describes the LDP associated to the empirical measure. The main
purpose of the present paper is to extend this result  by
considering also the empirical flow.

Below we will state two LDPs (Theorem \ref{LDP:misura+flusso} and
Theorem \ref{LDPteo2}) for the joint process given by the empirical
measure and flow.
In Theorem
\ref{LDP:misura+flusso} the  flow space is given by $L_1^+(E)$
endowed of the bounded weak* topology and, in order to have some
control at infinity in the case of infinite state space $V$,
compactness assumptions  are required.
  In Theorem \ref{LDPteo2}
the flow space is given by $[0,+\infty]^E$ endowed of the product
topology and  weaker assumptions are required (the same of
\cite{DV4}). On the other hand, the rate function has not always a
computable form.


 Let us now state precisely the compactness conditions under
which Theorem \ref{LDP:misura+flusso} holds (at least one of the
following  Conditions \ref{t:ccomp},  \ref{t:ccompls}
has to be satisfied). To this aim,
 given $f\colon V\to \bb R$ such that  $\sum_{y\in V }r(x,y) \, |f(y)| <+\infty$
for each $x\in V$,
 we
denote by $Lf \colon V\to \bb R$ the function defined by
\begin{equation}
  \label{Lf}
  L f\, (x) := \sum_{y\in V} r(x,y) \big[ f(y)-f(x)\big]
  ,\qquad x\in V.
\end{equation}

\begin{condition}
  \label{t:ccomp}
  There exists a sequence of functions $u_n\colon V \to (0,+\infty)$
  satisfying the following requirements:
  \begin{itemize}
  \item [(i)] For each $x\in V$ and $n\in\bb N$ it holds $\sum_{y\in V} r(x,y)
    u_n(y) <+\infty$. In the sequel $Lu_n \colon V \to \bb R$ is the
    function defined by \eqref{Lf}.
  \item [(ii)] The sequence $u_n$ is uniformly bounded from below.
    Namely, there exists $c>0$ such that $u_n(x)\ge c$ for any $x\in
    V$ and $n\in\bb N$.
  \item[(iii)] The sequence $u_n$ is uniformly bounded from above on compacts.
    Namely, for each $x\in V$ there exists a
    constant $C_x$ such that for any $n\in\bb N$ it holds $u_n(x)\le C_x$.
  \item [(iv)] Set $v_n :=  - Lu_n / u_n$. The sequence
    $v_n\colon V\to \bb R$ converges pointwise to some $v\colon V\to \bb R$.
  \item[(v)] The function $v$ has compact level sets.
    Namely, for each $\ell\in \bb R$
    the level set $\big\{x \in V \,:\, v(x)\leq \ell\big\}$ is
    finite.
  \item[(vi)]
    There exist a strictly positive  constant $\sigma$ and a positive
    constant $C$ such that
    $v \ge \sigma \, r - C$.
  \end{itemize}
\end{condition}

\begin{remark}
  Since $u_n >0$, it holds $v_n(x)= \sum _{y\in V} r(x,y) ( 1-
  u_n(y)/u_n(x))<r(x)$. Hence the function $v$ in Condition \ref{t:ccomp}
  must satisfy $v(x) \leq r(x)$ for all $x \in V$. Due to \emph{(v)}, this
  implies that also $r$ has compact level sets. In particular, when
  considering a Markov chain with infinite state space, the function
  $r$ must diverge at infinity.
\end{remark}

Replacing in Condition \ref{t:ccomp} the strictly positive constant
$\sigma$ with zero one obtains the same assumptions of Donsker and
Varadhan for the derivation in \cite{DV4} of the LDP for the
empirical measure of the Markov chain satisfying (A1)--(A4)
(shortly,  we will say that the \emph{Donsker--Varadhan condition} is
satisfied).
In particular,
the empirical measure satisfies a LDP with rate function
$\hat I \colon \mc P(V) \to [0,+\infty]$ given by
\begin{equation}
  \label{DVem}
   \hat I  (\mu)  =  \sup_{u>0}  \Big\{ -  \langle\mu,  Lu/u \rangle
   \Big\}.
\end{equation}
Under the same condition, the empirical process satisfies a LDP (see
Section~\ref{s:proiezione}).
Both these results still hold under a suitable compactness condition
concerning the hypercontractivity of the underlying Markov
semigroup, see 
\cite{DeS}.

With respect to the hypercontractity condition, in order to establish
the exponential tightness of the empirical flow we need extra
assumptions.  Recall that $\pi$ is the unique invariant measure of the
chain. The maps $P_t f(x):= \bb E( f(\xi^x_t ) )$, $t\in \bb R_+$,
define a strongly continuous Markov semigroup on $L^2(V, \pi)$. We
write $D_\pi$ for the Dirichlet form associated to the symmetric part
$(\mathcal L+\mathcal L^*)/2$ of the generator $\mathcal L$ in
$L^2(V,\pi)$. Since the time--reversed dynamics is described by a
Markov chain on $V$ with transition rates $r^*(x,y):=\pi (y) r(y,x)
/\pi (x)$, it holds
\begin{equation}
D_\pi (f)= \frac{1}{4} \sum _{x \in V} \sum_{y\in V} \bigl(
\pi(x)r(x,y)+ \pi(y) r(y,x) \bigr) \bigl( f(y)-f(x) \bigr)^2  \,,
\qquad f \in L^2 (V, \pi)\,.
\end{equation}
One can take this expression as definition of $D_\pi$, avoiding
all technicalities concerning   infinitesimal generators.
One says  that the Markov chain $\xi$ satisfies the  logarithmic
Sobolev inequality if there exists a constant $c_\mathrm{LS} \in
(0,+\infty)$ such that for any $\mu\in \mc P(V)$ it holds
(recall that $\pi(x)>0$  for any $x\in V$)
  \begin{equation}
    \label{ls}
    \Ent (\mu|\pi) \le c_\mathrm{LS} \,  D_\pi\left(\sqrt{\mu/\pi} \right),
  \end{equation}
where $\Ent(\mu|\pi)$ denotes the relative entropy of $\mu$ with
respect to $\pi$.

\begin{condition}
  \label{t:ccompls}
  $~~$
  \begin{itemize}
  \item[(i)]
    The Markov chain satisfies a logarithmic Sobolev inequality.
  \item[(ii)]
    The exit rate $r$ has an exponential moment with respect to the
    invariant measure.  Namely, there exists $\sigma>0$ such that
    $\big\langle \pi, \exp( \sigma \, r\big) \big\rangle <
    +\infty$.
  \item[(iii)]
    The graph $(V,E)$ is locally finite, that is for each vertex $y\in
    V$ the number of incoming and outgoing edges in $y$ is finite.
    \end{itemize}
\end{condition}

Item (iii) is here assumed for technical convenience and it should be
possible to drop it. Item (i) is the hypercontractivity condition
assumed in \cite{DeS} to deduce the Donsker-Varadhan theorem for the
empirical measure. Item (ii) is here required to prove the
exponential tightness of the empirical flow in $L^1_+(E)$.

\begin{remark}\label{piscina98}
By taking in \eqref{ls} $\mu=\delta_x$, Condition~\ref{t:ccompls}--\emph{(i)}
implies that $r$ has compact level sets.
\end{remark}

\subsection{LDP with flow space $L^1_+ (E)$ endowed of the bounded
  weak* topology}
\label{s:ldef}

We consider the space $L^1(E)$ equipped with the so-called
\emph{bounded weak* topology}. This is defined as follows. Recall
that the (countable) set $E$ is the collection of ordered edges in
$V$ with positive jump rate.  Let $C_0(E)$ be the collection of the
functions $F\colon E\to \bb R$ vanishing at infinity, that is the
closure of the functions with compact support in the uniform
topology. The dual of $C_0(E)$ is then identified with $L^1(E)$.
The weak* topology on $L^1(E)$ is the smallest topology
such that the maps $Q\in L^1(E)\to \langle Q,f\rangle\in \mathbb R$ with
$f\in C_0(E)$ are continuous.
Given $\ell>0$, let $B_\ell := \big\{ Q\in L^1(E) :\, \|Q\| \le
\ell\big\}$ be the closed ball of radius $\ell$ in $L^1(E)$
($\|\cdot\|$ being the standard $L^1$--norm). In view of the
separability of $C_0(E)$ and the Banach-Alaoglu theorem, the set
$B_\ell$ endowed with the weak* topology is a compact Polish space.
The bounded weak* topology on $ L^1(E)$ is then defined by declaring
a set $A\subset L^1(E)$ open if and only if $A\cap B_\ell$ is open
in the weak* topology of $B_\ell$ for any $\ell>0$. The bounded
weak* topology is stronger than the weak* topology (they coincide
only when $|E|<+\infty$) and for each $\ell>0$ the closed ball
$B_\ell$ is compact with respect to the bounded weak* topology. The
space $L^1(E)$ endowed with the bounded weak* topology is a locally
convex, complete linear topological space and a completely regular
space (i.e. for every closed set $C \subset L^1(E)$ and every
element $Q \in L^1(E) \setminus C$ there exists a continuous
function $f:L^1(E) \to[0,1]$ such that $f(Q)=1$ and $f(Q')=0$ for
all $Q'\in C$). Moreover, it is metrizable if and only if the set
$E$ is finite. We refer to \cite[Sec.\,2.7]{Me} for the proof of the
above statements and for  further details.


We regard $L^1_+(E)$ as a (closed) subset of $L^1(E)$ and consider it
endowed with the relative topology and the associated Borel
$\sigma$--algebra.
Accordingly, the empirical flow $Q_T$ will be considered as a
measurable map from $D(\bb R_+;V)$ to $L^1_+(E)$, defined $\bb P_x$
a.s., $x\in V$.
Recalling that we consider $\mc P(V)$, the set of probability measures
on $V$, with the topology of weak convergence, we finally consider the
product space $\mc P(V)\times L^1_+(E)$ endowed with the product topology
and regard the couple
$(\mu_T,Q_T)$ where $\mu_T$ is the empirical measure and $Q_T$ the
empirical flow, as a measurable map from $D(\bb R_+;V)$ to $\mc
P(V)\times L^1_+(E)$ defined $\bb P_x$ a.s., $x\in V$.

Below we state the LDP for the family  of probability measures on
$\mc P(V)\times L^1_+(E)$ given by $\big\{ \bb P_x \circ
(\mu_T,Q_T)^{-1} \big\}$ as $T\to+\infty$. Before stating precisely
the result, we introduce the corresponding rate function.  Let
$\Phi\colon \bb R_+ \times \bb R_+ \to [0,+\infty]$ be the function
defined by
\begin{equation}
  \label{Phi}
   \Phi (q,p)
   :=
   \begin{cases}
     \displaystyle{ q \log \frac qp - (q-p)}
     & \textrm{if $q,p\in (0,+\infty)$}
     \\
     \;p  & \textrm{if $q=0$, $p\in [0,+\infty)$}\\
     \; +\infty & \textrm{if $p=0$ and $q\in (0,+\infty)$.}
   \end{cases}
\end{equation}
Since $ \Phi (q,p) = \sup_{\lambda\in\bb R} \big\{ q\lambda -p(e^\lambda -1)
\big\}$, $\Phi$ is lower semicontinuos and convex.
We point out that, given $p>0$ and letting $N_t$, $t\in\bb R_+$ be a
Poisson process with parameter $p$, the sequence of real random variables
$\{N_T/T\}$ satisfies a large deviation principle on $\mathbb R$  with rate
function $\Phi (\cdot,p)$ as $T \to \infty$. This statement can be
easily derived from the G\"{a}rtner-Ellis theorem, see e.g\
\cite[Thm.~2.3.6]{DZ}. Recalling \eqref{divergenza_fluss} and
\eqref{Qmu}, we let $I\colon \mc P(V)\times L^1_+(E) \to
[0,+\infty]$ be the functional defined by
\begin{equation}
  \label{rfq}
  I (\mu,Q) :=
  \begin{cases}
    \displaystyle{
    \sum_{(y,z)\in E} \Phi \big( Q(y,z),Q^\mu(y,z) \big)
    }& \textrm{if  } \; \div Q =0\,,\; \langle \mu,r \rangle < +\infty
    \\
    \; +\infty  & \textrm{otherwise}.
  \end{cases}
\end{equation}
\begin{remark}\label{silente}
  In view of the lower semicontinuity and convexity of $\Phi$, $I$ is
  lower semicontinuous (apply Fatou lemma) and convex.
  Moreover, as proved in Appendix \ref{iobimbo},
  if $\langle \mu , r \rangle =+\infty$ the series in \eqref{rfq}
  diverges. Hence the condition $\langle \mu, r \rangle < +\infty$ can
  be removed from the first line of \eqref{rfq}.
\end{remark}

\begin{theorem}
  \label{LDP:misura+flusso}
  Assume the Markov chain satisfies (A1)--(A4) and at least one
  between Conditions \ref{t:ccomp} and   \ref{t:ccompls}.
  Then as $T\to+\infty$ the family of probability measures $\{\bb
  P_x\circ (\mu_T,Q_T)^{-1}\}$ on $\mc P(V)\times L^1_+(E)$
  satisfies a large deviation principle, uniformly for $x$ in
  compact
  subsets of $V$, with good and convex rate function $I$.
  Namely, for each not empty compact set $K \subset V$,
  each closed set $\mc C\subset \mc P(V)\times L^1_+(E)$,
  and each open  set $\mc A \subset \mc P(V)\times L^1_+(E)$, it
  holds
  \begin{align}
    \label{ubldp}
    & \varlimsup_{T\to+\infty}\;
    \sup_{x \in K} \;
    \frac 1T \log \bb  P_x \Big( (\mu_T,Q_T) \in \mc C \Big)
    \le -\inf_{(\mu,Q)\in \mc C} I(\mu,Q),
    \\
    \label{lbldp}
    & \varliminf_{T\to+\infty}\;
    \inf_{x \in K} \;
    \frac 1T \log \bb P_x \Big( (\mu_T,Q_T) \in \mc A \Big)
    \ge -\inf_{(\mu,Q)\in \mc A} I(\mu,Q).
   \end{align}
\end{theorem}

As discussed in Lemma \ref{basket}, under the assumptions in
Theorem~\ref{LDP:misura+flusso} it holds $\langle \pi, r\rangle <+
\infty$. In particular, $I(\mu,Q)=0$ if and only if
$(\mu,Q)=(\pi,Q^\pi)$. 
Hence, from the LDP one derives the law of large numbers for the
empirical flow in $L_+^1 (E)$, improving the pointwise version
discussed at the end of Section \ref{s:emef}. In addition, the
function $I$ has an affine structure:


 \begin{proposition}\label{gauchito}
Let  $(\mu , Q)\in \mc P(V) \times L^1_+(E)$ satisfy  $I(\mu,Q)<+\infty$. Then
\begin{itemize}
\item[(i)] All edges in the support $E(Q)$ of $Q$ connect
  vertices in the support of $\mu$, i.e.\ if $Q(y,z)>0$ then $y,z \in
  \supp (\mu) $.
\item[(ii)]
 Let $E^u(Q):=\bigl\{\,\{y,z\} \,:\, (y,z) \in E(Q) \text{ or } (z,y) \in
 E(Q)\bigr\}$.
  The oriented connected components of the oriented graph
  $(\supp (\mu), E(Q))$ coincide with the connected components
  of the unoriented graph  $(\supp (\mu), E^u(Q))$. 
 \item[(iii)]  $I(\mu,Q)$ has  the following affine
   decomposition. Consider the oriented graph $(\supp (\mu), E(Q))$
   and let $K_j$, $j \in J$, be the family of its oriented  connected
   components.  Consider the probability measure
  $\mu_j (\cdot) := \mu (\cdot | K_j)$
  and the flow $Q_j\in L^1_+(E)$ defined as
$$Q_j(y,z)= \begin{cases}
  \frac{Q(y,z)}{\mu(K_j) } & \text{ if } (y,z) \in E, \; y,z \in K_j \,,\\
  0 & \text{ otherwise}.
\end{cases}
$$
Then we have $(\mu, Q)=\sum_{j\in J}\mu(K_j)(\mu_j,Q_j)$ and
 \begin{equation}
I(\mu,Q) = \sum _{j \in J} \mu(K_j) I(\mu_j, Q_j) \,.
\end{equation}

\end{itemize}
\end{proposition}

For the unfamiliar reader, the definition of (oriented) connected
components is recalled after Remark \ref{zecchino}. Note that the
oriented components of $(\text{supp} (\mu), E(Q))$ coincide with the
irreducible classes of the Markov chain on $\text{supp}(\mu)$ with
transition rates $r(y,z):= Q(y,z)/\mu(y)$. Moreover, note that due to
Item (i) the graph $(\text{supp} (\mu), E(Q))$ is well defined. The
proof of Proposition \ref{gauchito} is given
in Section \ref{s:geometria}.

\subsection{LDP with flow space $[0,+\infty]^E$ endowed of the
product  topology}  \label{prodotto}

When considering the product topology on $[0,+\infty]^E$ we take
$[0,+\infty]$ endowed of the metric making the map $x \to
\frac{x}{1+x}\in [0,1]$ an isometry. Namely, on $[0,+\infty]$ we take
the metric $d(\cdot, \cdot)$ defined as $d(x,y)=\bigl| x/(1+x)
-y/(1+y)\bigr|$.  It is standard to define on the space
$[0,+\infty]^E$ a metric $D(\cdot, \cdot)$ inducing the product
topology: enumerating the edges in $E$ as $e_1,e_2, \dots$ we set $D(
Q, Q'):= \sum_{n=1}^{|E|} 2^{-n} d\bigl( Q(e_n), Q'(e_n)\bigr)$.

We write  $\mc M_S $ for the space of stationary probabilities on $D(\bb
R ;V)$ endowed of the weak topology.
 Given $R \in \mc M_S$ we denote
by $\hat \mu(R)\in \mc P (V) $ the marginal of $R$ at a given time and by $\hat
Q(R)$ the   flow in $[0,+\infty]^E$ defined as
 $\hat Q(R)(y,z):= \bb
E_{R}\bigl[ Q_T(y,z) \bigr]$ for all $(y,z) \in E$, where $\bb E_R$
denotes the expectation with respect to $R$. It is simple to check
that this expectation does not depend on the time $T>0$ (see Lemma \ref{pasquetta}). We point out that  jumps between a pair of states non belonging  to $E$ could take place with positive  $R$--probability. In particular, the flow $\hat Q(R)$ does not correspond to the complete flow associated to $R$.

\begin{lemma}
  \label{pasquetta}
  Given an edge  $(y,z) \in E$ and a stationary process $R \in
  \mc M_S$, the expectation $ \bb E_{R} \bigl[ Q_T (y,z) \bigr]\in
  [0,+\infty]$ does not depend on $T>0$.
\end{lemma}

\begin{proof}
Since $R$ is stationary, fixed  $t\in \bb R$ it holds $R( X_t \not =
X_{t-} )=0$. In particular, given $T>0$ and an integer $n$, $R$--a.e.\
it holds
$$
Q_T (X)\, (y,z)
=\frac{1}{n} \sum _{j=0}^{n-1}  Q_{ T/n} (\theta_{j T /n } X)  \, (y,z).
$$
Above we have used the notation $(\theta_s X)_t:= X_{s+t}$. From this
identity and the stationarity of $R$, taking the expectation
w.r.t. $R$ one gets $f(T)=f(T/n)$, where $f(T):=\bb E_R \bigl[ Q_T
(y,z) \bigr]$. Then  by standard arguments one gets that $f(T)=
f(1)$ as $T$ varies among the positive rational numbers.
 Since for $0<t_1 \leq T\leq t_2$ it holds $t_1 f(t_1) \leq
T f(T) \leq t_2 f(t_2)$ it is trivial to conclude that $f(T)$ is
constant as $T$ varies among the positive real numbers.
\end{proof}

We can now state our second main result:

\begin{theorem}\label{LDPteo2}
  Assume the Markov chain satisfies (A1)--(A4) together with the
  Donsker--Varadhan condition. Consider the space $\mc P(V)\times \bb
  [0,+\infty]^E$, with $ \mc P (V)$ endowed of the weak topology and
  $[0,+\infty]^E$ endowed of the product topology.  Then the following
  holds:
  \begin{itemize}
  \item[(i)] As $T\to+\infty$ the family of probability measures
    $\{\bb P_x\circ (\mu_T,Q_T)^{-1}\}$ on $\mc P(V)\times \bb
    [0,+\infty]^E$ satisfies a large deviation principle with good
    rate function
\begin{equation}
\tilde I (\mu,Q):= \inf \Big\{  H(R)\, :\, R \in \mc M_S\,,\;
 \hat \mu(R)=\mu\,,\;
 \hat Q(R)= Q  \Big\} \,.\label{samarcanda1}
  \end{equation}
Above  $H(R)$ denotes the entropy of $R$ with respect to the Markov
chain $\xi$ as defined in \cite{DV4}--(IV) (see Section
\ref{s:proiezione}). Moreover we have
\begin{equation}
\left\{
\begin{array}{ll}
 \tilde I(\mu,Q)=I(\mu,Q) & \text{ if }Q \in L^1_+ (E)\,,\\
 \tilde I(\mu,Q)=+\infty   & \text{ if } Q \not \in [0,+\infty)^E\,.
  \end{array}
  \label{samarcanda}
\right.
\end{equation}

\item[(ii)]
  If in addition Condition \ref{t:ccomp} is satisfied, then the rate
  function $\tilde I$ is given by
\begin{equation}\label{fatto!}
\tilde I(\mu,Q):= \begin{cases} I(\mu,Q) & \text{ if }
Q \in L^1_+(E)\,,\\
+\infty & \text{ otherwise}\,.
\end{cases}
\end{equation}
\end{itemize}
\end{theorem}
Since   Condition \ref{t:ccomp} implies  the
Donsker--Varadhan condition, Theorem \ref{LDPteo2}
 under  Condition \ref{t:ccomp}  implies the
variational characterization
$$ I (\mu,Q) = \inf \Big\{  H(R):\, R \in \mc M_S\,,\;
 \hat \mu(R)=\mu\,,\;
 \hat Q(R)= Q  \Big\}, \; (\mu,Q)
\in\mc P (V) \times L^1_+(E)\,.$$

In addition, note that \eqref{samarcanda} does not cover the case $Q
\in [0,+\infty)^E \setminus L_+^1(E)$.

\subsection{Outline}
The rest of the paper is devoted to the proofs of Theorems
\ref{LDP:misura+flusso} and \ref{LDPteo2}, and of Proposition \ref{gauchito}. Sections \ref{s:prep} and
\ref{s:geometria} contain preliminary results and the proof of
Proposition~\ref{gauchito}. Then in Section
\ref{s:diretto} we give a direct proof of Theorem
\ref{LDP:misura+flusso}. For this proof it is necessary to add the
condition that the graph $(V,E)$ is locally finite.

In Sections \ref{s:proiezione}, \ref{dim_cervo2} and \ref{dim_cervo3}
we remove the above condition and prove both Theorems
\ref{LDP:misura+flusso} and \ref{LDPteo2} by projection from the large
deviations principle for the empirical process proven by Donsker and
Varadhan in \cite{DV4}--(IV).
We discuss the details only for the Donsker-Varadhan type compactness
conditions. For this reason, we added item (iii) as a separate
requirement in the hypercontractivity type Condition~\ref{t:ccompls}.
By using similar arguments to the ones here presented, it should be
possible to remove it from Theorem~\ref{LDP:misura+flusso} and prove
the first statement in Theorem~\ref{LDPteo2} by assuming only items (i) and (ii) in
Condition~\ref{t:ccompls}.

Finally, in Section \ref{s:BD} we
discuss some examples from birth and death processes and compare the
different compactness conditions.

\section{Exponential estimates}\label{s:prep}
In this section we collect some preliminary results that will enter
in the proof of Theorems \ref{LDP:misura+flusso} and \ref{LDPteo2}.
Between other, we prove the exponential tightness in $L^1_+(E)$ of
the empirical flow when at least one between Conditions
\ref{t:ccomp}  and  \ref{t:ccompls} holds.

\subsection{Exponential local martingales}\label{exp_supm}
We start by comparing our Markov chain with a perturbed one. Let
$\hat{\xi}$ be a continuous time   Markov chain on $V$ with jump
rates $\hat r (y,z)$, $y\not=z$ in $V$. We assume that $\hat r(y):=
\sum _{z \in V} \hat r(y,z)< +\infty$ for all $y \in V$, thus
implying that the   Markov chain $\hat \xi$  is well defined at cost
to add a coffin state $\partial$ to the state space in case of
explosion \cite[Ch. 2]{N}.  We  write  $ \hat{ \bb P} _x$ for the
law on $D( \bb R_+, V \cup \{\partial\} )$ of the above Markov chain
$\hat \xi$ starting at $x\in V$.
We denote by $\rho_T$ the map  $\rho_T:D(\bb R_+, V \cup \{\partial
\}) \to D( [0,T], V\cup \{ \partial \} )$ given by restriction of
the path to the time interval $[0,T]$. We now assume that $\hat
r(y,z)=0$ if $(y,z) \not\in E$. Then, restricting the probability
measures $\bb P_x \circ \rho _T^{-1}$ and $ \hat {\bb P }_x \circ
\rho_T^{-1}$
to the set $D([0,T], V)$ (no explosion takes place in the interval
$[0,T]$), we obtain two reciprocally absolutely continuous measures
with Radon--Nykodim derivative  \begin{equation}\label{RN} \frac {d
\hat {\bb  P}  _{x} \circ \rho_T^{-1}   }{d\bb P_{x} \circ
\rho_T^{-1} } \Big|_{D([0,T], V)} =  \exp \left\{ - T\langle \mu_T,
\hat r -r\rangle\right\} \prod_{ (y,z) \in E} \left[\frac{ \hat
r(y,z)}{r(y,z)}\right] ^{T Q_T(y,z)}  \,.
\end{equation}
This formula can be checked very easily. Indeed, calling $\tau_1(X)<
\tau_2(X) < \tau_{N(X)}  (X)$ the jump times of the path $X$ in
$[0,T]$ (below $N(X)<+\infty$ almost surely) we have
\begin{multline*}
\bb P_x \circ \rho_T^{-1} \Big(N(X)= n\,, \; X(\tau_i)=x_i,
\; \tau_i \in (t_i, t_i+dt_i)\; \forall
i :1\leq i \leq n\Big)
\\ =
\Big[\prod _{i=0}^{n-1} e^{ -r(x_i) (t_{i+1}-t_i) } r(x_i,
x_{i+1})\Big] e^{-r(x_n) (T-t_n)}dt_1 \cdots dt_n\,,
 \end{multline*}
where $t_0:=0$ and $x_0:=x$, $0\leq t_1<t_2 < \cdots < t_n\leq T$,
$n=0,1,2,\dots$. Since a similar formula holds also for the law
$\hat {\bb  P}  _{x} \circ \rho_T^{-1}$, one gets \eqref{RN}.

\smallskip
As immediate consequence of the Radon--Nykodim derivative \eqref{RN}
we get the following result:
\begin{lemma}
  \label{t:em2}
  Let $F\colon E\to \bb R$ be such that $r^F(y) := \sum_z r(y,z) e^{F(y,z)}
  <+\infty$ for any $y\in V$.
  For $t\ge 0$ define $\bb M^F_t: D(\bb R_+, V) \to (0,+\infty)$ as
  \begin{equation}
    \label{expm2}
    \bb M^F_t :=
    \exp\Big\{ t \big[  \langle Q_t, F\rangle - \langle\mu_t, r^F-r\rangle\big]
    \Big\}
  \end{equation} where $\langle Q_t,F  \rangle = \sum _{(y,z) \in E} Q_t(y,z) F(y,z)$.
Then    for each $x\in V$ and $t\in \bb R_+$ it holds
  $\bb E_x \big( \bb M^F_t \big) \le 1$.
\end{lemma}
\begin{proof} By \eqref{RN} ($\hat r(y,z):=r(y,z) e^{F(y,z)}$),
$\bb E_x( \bb M_t ^F)= \hat{\bb P} _x \bigl(  D([0,t]; V )\bigr)\leq
1$.
\end{proof}

\begin{remark}
It is simple to check that the process $\bb M^F$ is a positive local
martingale and a supermartingale  with respect to   $\bb P_x$, $x\in
V$.
\end{remark}

\begin{remark}\label{willy}
Fixed $(y,z) \in E$, taking  in Lemma \ref{t:em2}  $F: = \pm \lambda
\delta _{y,z}$ with $\lambda >0 $ and applying Chebyshev inequality,
one gets for $\delta >0$ that the events $\{ Q_t(y,z) >\mu_t (y)
r(y,z) (e^\lambda-1) /\lambda +\delta\}$ and $\{ Q_t(y,z) <\mu_t (y)
r(y,z) (1-e^{-\lambda}) /\lambda -\delta\}$ have $\bb
P_x$--probability bounded by $e^{-t \delta\lambda}$. Using that $(e^{\pm \lambda} -1)/\lambda=\pm1+o(1)$  and  since
$\mu_t(y) \to \pi (y)$ as $t \to +\infty$   $\bb P_x$--a.s. by the ergodic theorem
\eqref{ergodico}, taking the limit $t \to+\infty$ and afterwards
taking $\delta, \lambda$ arbitrarily small, one recovers the LLN of
$Q_t(y,z)$
 towards $\pi(y)r(y,z)$ discussed in Section \ref{s:emef}.
\end{remark}

The next statement is deduced from Lemma \ref{t:em2}  by choosing
there $F(y,z) = \log [ u(z)/u(y)]$, $(y,z)\in E$ for some $u\colon V
\to (0,+\infty)$.

\begin{lemma}
  \label{t:em1}
  Let $u\colon V \to (0,+\infty)$ be such that $\sum_{z} r(y,z) u(z)
  <+\infty$ for any $y\in V$.  For $t\ge 0$ define $ \mc M^u_t   : D(\bb R_+, V) \to (0,+\infty)$ as
  \begin{equation}
    \label{expm1}
    \mc M^u_t :=  \frac{u(X_t)}{u(X_0)}
    \exp\Big\{ t \, \Big\langle \mu_t, - \frac{L u}{u}
    \Big\rangle\Big\}.
  \end{equation}
 Then for each $x\in V$
  and $t\in \bb R_+$ it holds $\bb E_x \big(\mc M^u_t \big) \le 1$.
\end{lemma}

\subsection{Exponential tightness}
\label{s:et}

We shall prove separately the exponential tightness of the
empirical measure and of the empirical flow. We first discuss the
case in which Condition~\ref{t:ccomp} holds. Then the proof of the
exponential tightness of the empirical measure is essentially a
rewriting of the argument in \cite{DV4} in the present setting. On
the other hand,  the proof of the exponential tightness of the
empirical flow depends on the extra assumption $\sigma>0$ in item
(vi) of Condition~\ref{t:ccomp}.

\begin{lemma}
  \label{t:letem}
  Assume Condition~\ref{t:ccomp} to hold and
  let the function $v$ and the constants $c,C_x,C,\sigma$ be as in Condition
  \ref{t:ccomp}.    Then for each $x\in V$ it holds
  \begin{equation}\label{sirena}
    \bb E_x \Big( e^{ T \langle\mu_T, v\rangle } \Big) \le
    \frac {C_x}{c} \,, \qquad
\bb E_x    \Big( e^{ T\sigma  \langle\mu_T, r\rangle } \Big) \le
    e^{TC}\frac {C_x}{c}\,.
  \end{equation}
\end{lemma}
\begin{proof} The second bound in \eqref{sirena} follows
trivially from the first one and item (vi) in Condition
\ref{t:ccomp}. To prove the first bound,
  let $u_n$ be the sequence of functions on $V$ provided by
  Condition~\ref{t:ccomp} and recall that  $v_n =  - Lu_n / u_n$.
  In view of the pointwise convergence of $v_n$ to $v$ and Fatou lemma
  \begin{equation*}
    \bb E_x \Big( e^{ T \langle\mu_T, v\rangle } \Big)
    \le \varliminf_n \:
    \bb E_x \Big( e^{ T \langle\mu_T, v_n\rangle } \Big)
    = \varliminf_n \: \bb E_x
    \Big( \exp\Big\{ T \,\Big\langle \mu_T, -\frac{L
      u_n}{u_n}\Big\rangle \Big\} \Big)
    \le \frac {C_x}{c}
  \end{equation*}
  where the last step follows from Lemma~\ref{t:em1} and
  items (ii)--(iii) in Condition~\ref{t:ccomp}.
\end{proof}

The following  provides the exponential tightness of the empirical
measure and the empirical flow.
\begin{proposition}
  \label{t:etem}
  Assume Condition~\ref{t:ccomp}.
  For each $x\in V$ there exists a sequence $\{\mc K_\ell\}$ of
  compacts in $\mc P(V)$ and a real sequence $A _\ell\uparrow +\infty$
  such that for any $\ell\in\bb N$
  \begin{align}
   &  \varlimsup_{T\to +\infty} \; \frac 1T
    \log \bb  P_x \big( \mu_T \not\in \mc K_\ell \big) \le -\ell\,,
    \label{arcobaleno1}\\
   &  \varlimsup_{T\to +\infty} \; \frac 1T
    \log \bb  P_x \big( \| Q_T\| >  A_\ell \big) \le -\ell\,.
\label{arcobaleno2}
      \end{align}
      In particular, the empirical measure and
      flow are exponentially tight.
\end{proposition}
\begin{proof} We first   prove
\eqref{arcobaleno1}.
  For a sequence $a_\ell\uparrow +\infty$ to be chosen later, set
  $W_\ell :=\{ x\in V :\, v(x) \le a_\ell\}$. In view of item (v) in
  Condition~\ref{t:ccomp}, $W_\ell$ is a compact subset of $V$. Set now
  \begin{equation*}
    \mc K_\ell := \bigcap_{m\ge\ell} \Big\{\mu\in\mc P(V)\,:\:
    \mu\big(W_m^\mathrm{c} \big) \le \frac 1{m} \Big\}
  \end{equation*}
  and observe that, by Prohorov theorem, $\mc K_\ell$ is a compact
  subset of $\mc P(V)$.

  From item (vi) in Condition~\ref{t:ccomp} (for this step we only
  need it with $\sigma=0$) and the definition of $W_\ell$ we deduce $v
  \ge a_\ell \id_{W_\ell^\mathrm{c}} - C$.  By the exponential
  Chebyshev inequality and Lemma~\ref{t:letem} we then get
  \begin{equation*}
    \begin{split}
      &\bb P_x \Big( \mu_T\big( W_\ell^\mathrm{c} \big) > \frac
      1{\ell}\Big) \le\bb P_x \Big( \langle\mu_T, v\rangle >
      \frac{a_\ell}{\ell}-C\Big)
      \\
      &\qquad\quad
      \le \exp\Big\{ -T \Big[ \frac{a_\ell}{\ell} -C\Big]
      \Big\} \: \bb E_x \Big( e^{ T \langle\mu_T, v\rangle} \Big)
      \le \frac{C_x}{c} \,
      \exp\Big\{ -T \Big[ \frac{a_\ell}{\ell} -C\Big]\Big\}.
    \end{split}
  \end{equation*}
  By choosing $a_\ell = \ell^2 +C\ell$ the proof is now easily concluded.

  Let us now prove \eqref{arcobaleno2}. By the second bound in Lemma
\ref{t:letem} and Chebyshev inequality,   $ \bb P_x \big(
\langle\mu_T, r \rangle > \lambda \big)
      \le \frac{C_x}{c} \: e^{ -T ( \sigma \lambda -C)} $ for any
      $\lambda >0$.
      In particular we obtain that
 $$ \bb P_x \big(
\langle\mu_T, r \rangle > A'_\ell  \big)
      \le
 \frac{C_x}{c} e^{-T\ell}\,,\qquad \quad A'_\ell:= \sigma^{-1}(\ell +
      C)\,.
 $$ Hence,
   it is enough to show that
  for each $x\in V$ there exists a sequence
  $A_\ell\uparrow +\infty$
   such
  that for any $T>0$ and any $\ell\in\bb N$
  \begin{equation}
    \label{enough1}
     \bb P_x \Big( \|Q_T\|> A_\ell \,,\,
     \langle\mu_T, r\rangle \le A_\ell' \Big)  \le   e^{-T\,\ell}\,.
  \end{equation}
  We consider the exponential local martingale of  Lemma~\ref{t:em2}
  choosing there $F\colon E\to \bb R$ constant, $F(x,y)=\lambda \in
  (0,+\infty)$ for any $(x,y)\in E$. We deduce
  \begin{equation*}
    \begin{split}
    & \bb P_x \Big( \|Q_T\|> A_\ell \,,\,\langle\mu_T, r\rangle\le A_\ell' \Big)
    \\
    & \qquad =
    \bb E_x \Big( e^{ -T \big[ \lambda \|Q_T\|
      - (e^\lambda-1) \langle\mu_T,r\rangle \big]}
    \: \bb M^F_T  \: \id_{\{\|Q_T\|> A_\ell\}}
    \, \id_{\{\langle\mu_T,r\rangle \le A_\ell'\}}
    \Big)
    \\
    & \qquad \le
    \exp\big\{ -T \big[ \lambda A_\ell -(e^\lambda-1) A_\ell'\big] \big\}
  \end{split}
  \end{equation*}
  where we used Lemma~\ref{t:em2} in the last step.
  The proof of \eqref{enough1} is now completed by choosing
  $A_\ell = \lambda^{-1} \ell + \lambda^{-1} (e^\lambda-1) A_\ell'$.

\smallskip

 Recalling that the closed ball in $L^1_+(E)$ is
compact with respect to the bounded weak* topology, the exponential
tightness of the empirical flow is due to \eqref{arcobaleno2}.
\end{proof}

\bigskip
We next discuss the exponential tightness when
Condition~\ref{t:ccompls} is assumed.

\begin{proposition}
  \label{t:extls}
  Fix $x\in V$. If item (i) in Condition~\ref{t:ccompls} holds then
  the sequence of probabilities $\{\bb P_x\circ \mu_T^{-1}\}$ on $\mc
  P(V)$ is exponentially tight. If furthermore it holds also item (ii)
  in Condition~\ref{t:ccompls}, then the sequence of probabilities
  $\{\bb P_x\circ Q_T^{-1}\}$ on $L^1_+(E)$ is exponentially tight.
\end{proposition}

While the first statement is a consequence of the general results in
\cite{DeS}, we next give a direct and alternative  proof also of
this result. We premise an elementary lemma whose proof is omitted.
\begin{lemma}
  \label{t:lextls}
  Let $\pi\in\mc P(V)$ be such that $\pi(x)>0$ for any $x\in V$.
  There exists a decreasing function $\psi_\pi\colon (0,1)\to (0,+\infty)$
  such that $\lim_{s\downarrow 0} \psi_\pi(s) =+\infty$ and
  \begin{equation*}
    \sum_{x\in V} \pi(x) \, \psi_\pi(\pi(x)) < +\infty.
  \end{equation*}
\end{lemma}


\begin{proof}[Proof of Proposition~\ref{t:extls}]
  We prove first the exponential tightness of the empirical
  measure.
  Let $\pi$ be the invariant measure of the chain,
  $\psi_\pi$ be the function provided by Lemma~\ref{t:lextls}
  and $\alpha:=\sum_x \pi(x)\psi_\pi(\pi(x))<+\infty$.
  We define
  $v \colon V\to (0,+\infty)$ as
  \begin{equation*}
    v(x) := \log \frac{\psi_\pi \big(\pi(x)\big)}{\alpha} \,,\qquad x\in V.
  \end{equation*}
  Then, in view of
  Lemma~\ref{t:lextls}, $v$ has compact level sets and
  $\big\langle\pi,e^v\big\rangle=1$.

  By the proof of Proposition~\ref{t:etem}, it is enough to show the
  following bound.  For each $x\in V$ there exist constants $\lambda,
  C_x>0$ such that for any $T>0$
  \begin{equation}
    \label{exbbis}
    \bb E_x \Big( e^{\lambda T\, \langle\mu_T,v\rangle }\Big) \le C_x.
  \end{equation}

Since the function $v$ diverges at infinity, it is bounded from below
and has finite level sets $V_n:=\{x \in V\,:\,v(x) \leq n\}$. We
define $v_n(x) := v(x) \id_{x \in V_n}$ and set for $x,y \in V$
 $$ r_n (x,y) := \begin{cases}
 r(x,y) & \text{ if } x  \in V_n  \\
r(x,y)/r(x)  & \text{ if } x \not \in V_n
\end{cases}
\qquad \pi_n (x) := \begin{cases}\frac{\pi(x)}{Z_n} & \text{ if } x
\in V_n \\
\frac{ \pi (x)r (x)  }{Z_n} & \text{ if } x \not \in V_n
\end{cases}
$$
where 
 $Z_n$ is the
normalizing constant making $\pi _n$ a probability measure on $V$.
Due to Condition \ref{t:ccompls} it holds $\langle \pi, r \rangle <
+\infty$ , thus implying that $Z_n$ is well defined  and that $\lim
_{n\to \infty} Z_n = 1$.

We notice that  the function $r_n\colon V \to (0,+\infty)$,
$r_n(x):= \sum _{y \in V} r_n(x,y)$, is bounded from above. We then
consider the continuous--time Markov chain $\xi^{(n)}$    in $V$
with transition rates $r_n(\cdot, \cdot)$. Since $\pi_n(x) r_n(x,y)=
\pi (x) r(x,y) /Z_n$,  we derive that $\pi_n$ is the unique
invariant distribution of $\xi^{(n)}$. We denote by $\bb E_x ^{(n)}$
the expectation w.r.t. the law of the Markov chain $\xi^{(n)}$
starting at $x$ and by $\mc A_n$ the subset of $ D([0,T];V)$ defined
as $\mc A _n= \{X: X_t \in V_n \, \forall t \in [0,T]\}$. Then we
have
\begin{equation}\label{situazione}
 \bb E_x \Big( e^{\lambda T\, \langle\mu_T,v\rangle } \Big)= \lim
 _{n\to \infty} \bb E_x \Big( e^{\lambda T\, \langle\mu_T,v_n\rangle }\id_{\mc A_n}
 \Big)= \lim
 _{n\to \infty} \bb E^{(n)} _x \Big( e^{\lambda T\, \langle\mu_T,v_n\rangle }\id_{\mc A_n}
 \Big)
\end{equation}
(the first identity follows from the monotone convergence theorem).
Since $v_n$ and $r_n$ are  bounded function, we can apply
\cite[App.~1, Lemma~7.2]{KL} and get
  \begin{equation}
    \label{camminarebis}
    \begin{split}
    & \bb E^{(n)} _x \Big( e^{\lambda T\, \langle\mu_T,v_n\rangle }\id_{
      \mc A_n}
    \Big)  \leq  \frac{1}{\pi_n(x)} \bb E^{(n)}_{\pi_n} \Big( e^{\lambda T\,
      \langle\mu_T,v_n\rangle } \Big)
    \\
    & \qquad \leq
    \frac{1}{\pi_n(x)} \exp\Big\{ T \sup_{f \,: \: \pi_n(f^2)=1 }
    \big[ - D^{(n)}_{\pi_n} (f)
    + \lambda \,\langle\pi_n, f^2 \, v_n\rangle \big] \Big\}\,.
    \end{split}
  \end{equation}
  Since $v_n$ vanishes on $V_n^\mathrm{c}$ and $v_n \leq v$ we have
  $\langle \pi_n, f^2 v_n\rangle
  \leq Z_n^{-1} \langle\pi,f^2 v\rangle$, while from the identity  $\pi_n(x)
  r_n(x,y)= \pi (x) r(x,y)
  /Z_n$ we get $D^{(n)}_{\pi_n} (f ) =  Z_n^{-1} D _{\pi} (f)$.
  Hence
  \begin{equation}\label{ercolino}
   - D^{(n)}_{\pi_n} (f)
    + \lambda \,\langle\pi_n, f^2 \, v_n\rangle
    \le \frac{1}{Z_n}
    \big[ - D_{\pi} (f)
    + \lambda \,\langle\pi, f^2 \, v\rangle \big]\,.
  \end{equation}
We next show that if $\lambda \in (0,1/c_{\mathrm{LS}})$ then the
right hand side in \eqref{ercolino} is bounded above by zero whenever $\pi(f^2)
<+\infty$ (note that,
in view of
Remark~\ref{piscina98} and  since  $\pi_n(f^2)=1$, it holds  $\pi(f^2)
<+\infty$). To this aim let $f_* := f / \sqrt{ \pi(f^2)}$, hence $\pi ( f_*^2)=1$.
The
basic entropy inequality yields
\begin{equation*}
  \langle \pi,  f_* ^2 v \rangle \le \log \big\langle\pi,e^{v}\big\rangle +
  \Ent (\mu|\pi), \qquad \mu = f_*^2 \pi.
\end{equation*}
Recalling that $\langle\pi,e^{v}\rangle=1$, the logarithmic Sobolev
inequality \eqref{ls}  implies that $\langle \pi,  f_* ^2 v \rangle \leq c_{\rm LS} D _\pi( f_*)$, hence our claim. The bound \eqref{exbbis}
follows, thus concluding the proof of the exponential tightness of the
empirical measure.

  To prove the exponential tightness of the empirical flow,
  we first observe that the only properties of $v$ used to derive
  \eqref{exbbis} are that $v$ has compact level sets and satisfies
  $\langle\pi, e^v\rangle =1$. By Remark \ref{piscina98} and
  item (ii) in Condition \ref{t:ccompls}, the function
  $\tilde v = \sigma r - \log\langle \pi, e^{\sigma r}\rangle $ meets these
  two requirements. Hence the bound \eqref{exbbis} holds for $\tilde
  v$ which implies
  \begin{equation}
    \label{sirena20}
    \bb E_x    \Big( e^{ T\sigma'  \langle\mu_T, r\rangle } \Big) \le
    e^{TC} {C_x} .
  \end{equation}
  for some $\sigma', C,C_x >0$.
   In view of \eqref{exbbis} and \eqref{sirena20}, the proof of the
  exponential tightness of the empirical flow is achieved by the
  argument leading to \eqref{arcobaleno2}.
\end{proof}

We conclude with a simple observation on the stationary flow:
\begin{lemma}\label{basket}
Assume at least one between Conditions \ref{t:ccomp} and
\ref{t:ccompls} to hold. Then $\langle \pi,r\rangle < +\infty$,
equivalently $Q^\pi \in L^1_+(E)$.
\end{lemma}
\begin{proof} The thesis is trivially true
under Condition \ref{t:ccompls}.  Let us assume Condition
\ref{t:ccomp}. By Lemma \ref{t:letem} we have $\bb
E_x    \big( e^{ T\sigma \langle\mu_T, r\rangle } \big) \le
    e^{TC}C_x/c$.
 We restrict to
$V$ infinite, the finite case  being obvious. Enumerating the points
in $V$ as $\{x_n\}_{n \geq 0}$,  by the ergodic theorem
\eqref{ergodico} fixed $N$ there exists a time $T_0=T_0(N)>0$ and a
Borel set $\mc A \subset D( \bb R_+;V)$ such that (i) $\bb P_x(\mc
A) \geq 1/2$ and (ii)  $ \mu_T(x_n) \geq \pi (x_n)/2$  for all $T
\geq T_0$ and $n \leq N$ $\bb P_x$--a.s. on $\mc A$. Hence, for all
$T \geq T_0$ it holds
$$ e^{ T \sigma  \sum _{n=0}^N \pi (x_n)r(x_n) /2}/2\leq \bb E_x \big(
e^{T \sigma  \sum _{n=0}^N \mu_T (x_n)r(x_n)}; \mc A\big) \leq \bb
E_x \big( e^{ T\sigma \langle\mu_T, r\rangle } \big) \le
  e^{TC}C_x/c .
$$
This implies that $\sum _{n=0}^N \pi (x_n)r(x_n)\leq 2C/\sigma$. To
conclude it is enough to take the limit $N \to +\infty$.
\end{proof}

\section{Structure of divergence--free flows in
  $L^1_+(E)$}
\label{s:geometria}

In this section we show that any divergence--free flow in $L^1_+(E)$,
and more in general any divergence--free flow in $\bb R_+^E$ with zero
flux towards infinity, can be written as superposition of flows along
self avoiding finite cycles.  See \cite{GV} for other problems related
to cyclic decompositions of divergence--free flows on graphs and
\cite{Smi} for similar decompositions for divergence--free vector
valued measures on $\mathbb R^d$.

We first introduce some key graphical structures.  A \emph{finite
  cycle} $C$ in the oriented graph $(V,E)$ is a sequence
$(x_1,\ldots,x_k)$ of elements of $V$ such that $(x_{i},x_{i+1})\in E$
when $i=1,\dots, k$ and the sum in the indices is modulo $k$.  A
finite cycle is \emph{self avoiding} if for $i\neq j$ it holds
$x_i\neq x_j$. Given $(x,y)\in E$, if there exists an index
$i=1,\dots,k$ such that $(x,y)=(x_i,x_{i+1})$
we write $(x,y)\in C$. Similarly,   given $x\in V$, if there exists
an index $i=1,\dots, k$ such that $x=x_i$ we say that $x\in C$. The
collection of all the self avoiding finite cycles in $(V,E)$ is a
countable set which we denote by $\mc C$.  In the sequel we shall
mostly regard elements $C\in\mc C$ as finite subsets of $E$ and
denote by $|C|$ the corresponding cardinality. Consider an invading
sequence $V_n \nearrow V$  of finite subsets $V_n$. This means a
sequence such that $|V_n|<+\infty$, $V_n\subset V_{n+1}$ and
moreover $\cup_nV_n=V$. For any fixed $n$ we define
\begin{equation}\label{nevischio}
E_n:=\left\{(y,z)\in E\,:\, y,z\in V_n\right\}\,, \end{equation} and
observe that it is an invading sequence of edges. Given a flow $Q\in
\bb R_+ ^E$, we define
\begin{align}
&E(Q):=\left\{(y,z)\in E\,:\, Q(y,z)>0\right\}\,, \label{freddo1}\\
& M_n(Q):=\max_{(y,z)\in E_n}Q(y,z)\,, \label{freddo2}\\
& \phi_n^{+}(Q):=\sum_{y\in V_n, z\not\in V_n}Q(y,z)\,,\label{freddo3}\\
& \phi_n^{-}(Q):=\sum_{y\not\in V_n, z\in
V_n}Q(y,z)\,.\label{freddo4}
\end{align}
The definition \eqref{divergenza_fluss} of the divergence of a flow
$Q$ is well posed also if $Q\not\in L_+^1(E)$ provided the incoming
and outgoing fluxes are finite at every vertex.
In this case the series in \eqref{freddo3} and \eqref{freddo4} are
convergent.
By a divergence--free flow $Q\in \bb R_+^E$ we mean that $Q$  has well
defined vanishing divergence.
Moreover, we say that $Q$  has zero flux towards infinity if
there exists an invading sequence  $V_n  \nearrow
V$  of finite subsets $V_n$ such that
\begin{equation}
\lim_{n\to +\infty} \phi_n^{+}(Q)=0\,. \label{fluflu}
\end{equation}
Finally, we  say that $Q$ admits a cyclic decomposition if there are
constants  $\hat Q(C)\ge 0$, $C\in\mc C$
  such that
  \begin{equation}
    \label{c:2}
    Q = \sum_{C\in\mc C} \hat Q(C)\id_C\,.
  \end{equation}
  Namely, for each $(y,z)\in E$ it holds
 $   Q (y,z) =\sum_{
 C\in\mc C,\,
 C\ni (y,z)} \hat Q(C)$.
We emphasize that the constants  $\hat Q(C)$, $C\in\mc C$, are
not uniquely determined by the flow $Q$.

\begin{lemma}
  \label{t:cicli}
  Let $Q\in \bb R _+^E$ be a divergence--free flow having zero flux
  towards infinity. Then $Q$ admits a cyclic decomposition
  \eqref{c:2}. In particular, any divergence--free flow $Q \in
  L^1_+(E)$ has a cyclic decomposition.
\end{lemma}

\begin{proof}
Since  \eqref{fluflu} holds for any invading sequence of vertices if
 $Q\in L^1_+(E)$,  the second statement  follows directly
from the former on which we concentrate.

On a finite graph  any divergence--free flow admits a cyclic
decomposition. The proof follows classical arguments (see e.g.
\cite{GV,MQ}). If $Q$ has finite support, i.e. if $|E(Q)|<+\infty$,
the thesis follows directly by the analogous result on finite
graphs. We will then consider only the case of infinite support,
using below the result in the finite case. Let $V_n$ be
an invading sequence satisfying \eqref{fluflu}.

We assume $|E(Q)|=+\infty$ and  $\div Q=0$. Due to the zero
divergence condition,  a discrete version of the Gauss theorem
guarantees that $\phi_n^+(Q)=\phi_n^-(Q)$. We define by an iterative
procedure a sequence of flows $Q^i $, $i \geq 0$,  with infinite
support and having zero flux towards infinity as follows. We set
$Q^0:=Q$ and explain how to define $Q^{i+1}$ knowing $Q^i$. First,
we define $n_i:=\inf\left\{n\in \mathbb N\,:\,
M_n(Q^i)>\phi^+_n(Q^i)\right\}$. Since $Q^i\neq 0$, it must be
$n_i<+\infty$. Indeed,  $\phi^+_n(Q^i)$ is a sequence in $n$
converging to zero, while  $M_n(Q^i)$ is a non decreasing sequence
not identically zero. Let $g$ be a ghost site and define the  flow
$Q^i_g$ on a finite graph having vertices $V_{n_i}\cup\{ g\}$ as
$$
\left\{
\begin{array}{lr}
Q^i_g(y,z):=Q^i(y,z)\,, & (y,z)\in E_{n_i}\,,\\
Q^i_g(y,g):=\sum_{z\not\in V_{n_i}}Q^i(y,z)\,, & y\in
V_{n_i}\,,\\
Q^i_g(g,y):=\sum_{z \not\in V_{n_i}}Q^i(z,y)\,, & y\in
V_{n_i}\,.\\
\end{array}
\right.
$$
Roughly speaking, the flow $Q^i_g$ is obtained from $Q^i$ by
collapsing all vertices outside $V_{n_i}$ into a single vertex,
called $g$. By construction we have $\div  Q^i_g=0$. Calling
$\mathcal C_{n_i}^g$ the collection of self avoiding cycles of the
finite graph and using the validity of the cyclic decomposition in
the finite case, we have
\begin{equation}
Q^i_g=\sum_{C\in \mathcal C_{n_i}^g }\hat Q^i_g(C)\id_C\,.
\label{decfin}
\end{equation}
We claim that in the decomposition \eqref{decfin} there exists a
self avoiding cycle $C_i$ not visiting the ghost site $g$ and such
that $\hat Q^i_g(C_i)>0$.
Let us suppose by contradiction that our claim is false and let
$(x^*,y^*)\in E_{n_i}$ be  such that $Q^{i}(x^*,y^*)=M_{n_i}(Q^i)$.
Then we have
$$
M_{n_i}(Q^i)=Q^{i}(x^*,y^*)=\sum_{C\in \mathcal C_{n_i}^g }\hat
Q^i_g(C) \id _{(x^*,y^*)\in C}\leq\sum_{C\in \mathcal C_{n_i}^g
}\hat Q^i_g(C)=\phi_{n_i}^+(Q^i)\,.
$$
The last equality follows by the fact that any cycle
with positive weight in $\mathcal C_{n_i}^g$ has to contain necessarily
the ghost site $g$.
This contradicts the definition of $n_i$, thus proving our claim.
\smallskip

At this point, we know that there exists a self avoiding cycle
$C_i:=(x_1, \dots ,x_k)$ such that $x_j\in V_{n_i}$ and
$Q^i(x_j,x_{j+1})>0$ for any $j$ (the sum in the indices is modulo
k). We fix $m_i:=\min_{j=1,\dots,k}Q^i(x_j,x_{j+1})$ and define
$$
Q^{i+1}:=Q^i-m_i\id_{C_i}= Q-\sum_{j=0}^im_j\id_{C_j} \,.
$$
With this definition we have that $Q^{i+1}$ is 
an element of $\bb R_+^E$, it satisfies $\div Q^{i+1}=0$, it has zero
flux towards infinity, and infinite support. Moreover
\begin{equation}
|E_{n_i}\cap E(Q^{i+1})|\leq |E_{n_i}\cap E(Q^i)|-1\,.\label{dimi}
\end{equation}
Condition \eqref{dimi} implies that $\lim_{i\to
+\infty}n_i=+\infty$. Hence, fixed any $(y,z)\in E$, for $i$ large
it holds
\begin{equation}
Q^i(y,z)\leq M_{n_i-1}(Q^i)\leq \phi^+_{n_i-1}(Q^i)\leq
\phi^+_{n_i-1}(Q) \label{stima}
\end{equation}
(for the first inequality note that $(y,z) \in E_{n_i-1}$ for $i$
large, for the second one use the definition of $n_i$, for the third
one observe that by construction $Q^i \leq Q$).

Since the right hand side of \eqref{stima} converges to zero when $i$
diverges we obtain $\lim_{i\to +\infty} Q^i(y,z)=0$ for any $(y,z)\in
E$. Finally we get
$$
\lim_{i\to +\infty}\Big(Q(y,z)-\sum_{j=0}^im_j\id_{C_j}(y,z)\Big)=
\lim_{i\to +\infty}Q^{i+1}(y,z)=0\,.
$$
This trivially implies that $Q=\sum_{j=0}^{\infty}m_j\id_{C_j}$.
\end{proof}

\begin{remark}\label{zecchino}
It is easy to see that Lemma \ref{t:cicli} remains valid  if the
condition of zero flux towards infinity is satisfied just by the
reduced flow $q\in \bb R_+^E$ defined as
$$
q(y,z):=\left\{
\begin{array}{ll}
Q(y,z) & (z,y)\not \in E\,,\\
Q(y,z)-\min\left\{Q(y,z),Q(z,y)\right\}& (z,y)\in E\,.
\end{array}
\right.
$$
\end{remark}

Given an oriented graph $(\mc V,\mc E)$  with countable $\mc V,\mc
E$ we say that it is connected if for any $y,z\in \mc V$ there exist
$x_1,\dots,x_n$ such that $x_1=y$, $x_n=z$ and $(x_i,x_{i+1})\in \mc
E$, $i=1,\dots, n-1$. To every oriented graph we can associate an
unoriented graph $(\mc V,\mc E^u)$ for which $\left\{y,z\right\}\in
\mc E^u$ if at least one among $(y,z)$ and $(z,y)$ belongs to $\mc
E$. We say that the  unoriented graph  $(\mc V,\mc E^u)$    is
connected if for any $y,z\in \mc V$ there exist $x_1,\dots,x_n$ such
that $x_1=y$, $x_n=z$ and $\left\{x_i,x_{i+1}\right\}\in \mc E^u$,
$i=1,\dots, n-1$.

The following lemma will be useful.

\begin{lemma}
  Let $(\mc V,\mc E)$ be a countable oriented graph.
  Then $(\mc V,\mc E)$ is connected if and only if \emph{(i)}
  $(\mc V,\mc E^u)$ is connected and \emph{(ii)}
  there exists a flow $Q\in L^1_+(\mc E)$ with $Q(y,z)>0$
  for any $(y,z)\in \mc E$ and $\div Q=0$.
\label{tommy}
\end{lemma}

\begin{proof}
Suppose first that $(\mc V,\mc E)$ is connected. Then
$(\mc V,\mc E^u)$ is also trivially connected. To show property (ii),
since $\mc C$ is countable. we can find a sequence $\{\alpha_C, \, C\in \mc C\}$
with $\alpha_C >0$ and $\sum_{C} \alpha_C < +\infty$. We then define
$Q =\sum_{C}  \alpha_C \id_C$ which is summable and
divergence--free. It remains to check that  $Q(y,z)>0$
for any $(y,z)\in \mc E$. Since $(\mc V,\mc E)$ is connected, we can
add to $(y,z)$ an oriented path from $z$ to $y$ obtaining a finite cycle
$C\ni (y,z)$.

We now prove the converse implication. In view of Lemma \ref{t:cicli},
the flow $Q$ in (ii) admits a cyclic decomposition \eqref{c:2}. If
$(y,z)\in \mc E$ then $0< Q(y,z)= \sum_{C\ni (y,z)} \hat Q(C)$. Thus
there exists a finite cycle $C$ containing $(y,z)$, hence there exists
an oriented path from $z$ to $y$. This shows that neighbors in $(\mc
V, \mc E^u)$ are connected in $(\mc V, \mc E)$.
\end{proof}

We can now give the proof of Proposition \ref{gauchito}.

\begin{proof}[Proof of Proposition \ref{gauchito}] $~$

\noindent
  Proof of (i). Fix $(y,z)\in E$ with $Q(y,z)>0$.
  From the definition of $I(\mu,Q)$ and $\Phi$ we deduce that
  $\mu(y) r(y,z)>0$ and therefore $\mu(y)>0$. Since $\div Q=0$
  and the ingoing flow in $z$ is strictly positive then there exists
  $(z,y')\in E$ with $Q (z,y') >0$ hence, by what just proven, $\mu(z)>0$.

\noindent
  Proof of (ii). It  is an immediate
  consequence of Lemma \ref{tommy}.

\noindent
  Proof of (iii).  To
  this aim we first observe that $\div Q_j=0$. Indeed, the following
  property (P) holds: given $y \in V$, if $Q(y,z)>0$ or $Q(z,y)>0$
  then $z $ belongs to the
  same oriented connected component of $y$
  (apply Item (ii)).  This property and the zero divergence of $Q$
  imply that $\div Q_j=0$.
  By definition \eqref{rfq} and Remark \ref{silente},
\begin{equation}\label{ritenzione}
 I(\mu_j, Q_j) = \sum _{(y,z) \in E \cap (K_j \times K_j)}
 \Phi\bigl(Q_j (y,z), Q^{\mu_j}(y,z) \bigr)+
\sum _{(y,z)\in E\cap (K_j \times K_j^c) }  Q^{\mu_j}(y,z)
\,.
\end{equation}
Always property (P) implies
that
\begin{equation}\label{biblio} I(\mu,Q)= \sum_j \left\{ \sum _{(y,z)
      \in E \cap( K_j \times K_j)} \Phi\bigl(Q(y,z), Q^\mu(y,z) \bigr)
    + \sum _{(y,z)\in E\cap (K_j  \times K_j^c) }  Q^\mu (y,z) \right\}.
\end{equation}
To conclude compare \eqref{ritenzione} with \eqref{biblio} using that
$Q(y,z)= \mu(K_j) Q_j(y,z)$ and
$Q^\mu(y,z)= \mu(K_j) Q^{\mu_j} (y,z)$  if $(y,z) \in  E$ with $y \in K_j$.
 \end{proof}

\subsection{An approximation result for the function
$I(\mu,Q)$}
\label{lattino}

Let $\mc S$ be the subset of $ \mc P(V)\times L^1_+(E)$ given by the
elements $(\mu,Q)$ with $I(\mu,Q)<+\infty$ and  such that the graph
$(\supp(\mu), E(Q) )$ is finite   and connected.

\begin{proposition}
  \label{t:dt}
  Fix $(\mu,Q)\in \mc P(V)\times L^1_+(E)$. There exists a sequence
  $\{(\mu_n,Q_n)\}$ in $\mc S$ such that $(\mu_n,Q_n)\to (\mu,Q)$
  in $\mc P(V)\times L^1_+(E)$ and
  \begin{equation}
    \label{dt}
    \varlimsup_{n\to + \infty} \: I(\mu_n,Q_n) \le I(\mu,Q)\,.
  \end{equation}
\end{proposition}
As proven below, the convergence $(\mu_n,Q_n)\to (\mu,Q)$ in $\mc
P(V)\times L^1_+(E)$ holds also with $L^1_+(E)$ endowed with the
$L^1$--norm (strong topology).

\begin{proof}
  We consider only elements $(\mu,Q)$ such that $I(\mu,Q)<+\infty$,
  otherwise the thesis is trivially true. In particular, 
  $\div Q=0$.
  Denote by $\mathcal S^*$ the set of elements $(\mu,Q) \in \mathcal
  P(V)\times L^1_+(E)$ with finite support (i.e.\ with finite
  $\supp(\mu)$ and $E(Q)$) and $\div Q=0$.
  We first show that \eqref{dt} holds for $(\mu,Q)\in\mathcal S^*$

  Let $(\mu,Q)\in \mc S^*$. Then there exists a
  finite connected oriented subgraph $(V^*,E^*)$ of $(V,E)$ which contains
  $(\supp(\mu), E(Q))$ (add to $(\supp(\mu), E(Q))$
  suitable paths joining the connected components of  $(\supp(\mu),
  E(Q))$\,). Denote by $r^*$ the restriction of $r$ to $E^*$ and let
  $\pi^*$ be the (unique) invariant probability of the chain with rates
  $r^*$ on the graph $(V^*,E^*)$. Set also $Q^*(y,z) := \pi^*(y)
  r^*(y,z)$ and extend $\pi^*$, $Q^*$ to functions on $V$, $E$ by
  setting them equal to zero outside $V^*$, $E^*$.
  Due to the invariance of $\pi^*$,   $\div Q^*=0$.
  Moreover, it holds
  \begin{equation*}
    I(\pi^*,Q^*) = \sum_{(y,z)\in E^*} \Phi(Q^*(y,z), \pi^*(y) r(y,z))
    +   \sum_{(y,z)\not\in E^* }\pi^*(y) r(y,z)\,.
  \end{equation*}
  As the first sum is a finite sum of finite terms and the second one
  is bounded by $\langle\pi^*,r\rangle$, we deduce $I(\pi^*,Q^*)<+\infty$.
  We define the sequence
  $\big\{(1-\frac 1n )(\mu, Q) +  \frac 1n (\pi^*,Q^*)\big\}$ which belongs
  to $\mathcal S$ and converges to  $(\mu,Q)$ (even
  with $L^1_+(E)$ endowed of the strong topology).
  By the convexity of $I$ stated in Remark \ref{silente}
$$
\varlimsup_{n \to \infty}I\Bigl(
(1- \frac 1n) (\mu,Q)+\frac 1n (\pi^*,Q^*) \Bigr)\le I(\mu,Q)\,.
$$

  Given $(\mu,Q) \in \mathcal P(V)\times L^1_+(E)$ with $\div Q=0$,
  we now show that there exists a sequence $\{(\mu_n, Q_n )\}\subset \mc
  S^*$ such that \eqref{dt} holds. The thesis then follows by a
  diagonal argument.
  We fix $(\mu,Q)$ with $\div Q=0$ and $I(\mu,Q)<+\infty$.  By
  Lemma \ref{t:cicli} the cyclic decomposition \eqref{c:2} of $Q$
  holds.  We fix an invading sequence $V_n\nearrow V$ of finite
  subsets and call $E_n$ the edges in $E$ connecting vertices in $V_n$
  (recall \eqref{nevischio}).  Finally, we construct the sequence
  $(\mu_n,Q_n)\in \mathcal S^*$ by
$$
\mu_n:=\frac{\mu|_{V_n}}{\mu(V_n)}\,, \qquad \qquad
Q_n:=\sum_{\left\{C\in \mathcal C\,:\, C\subset E_n\right\}}\hat
Q(C)\id_C\,.
$$
For $n$ large $\mu(V_n)>0$ and the definition is well posed. Clearly
$(\mu_n,Q_n)$ converges to $(\mu,Q)$  (also considering the strong
topology of  $L^1_+(E)$).
 It remains to show \eqref{dt}. By construction $\div Q_n=0$ and $\langle \mu_n, r \rangle <+\infty$,
  hence, recalling \eqref{rfq},
  \begin{equation}
    \label{sum}
    I(\mu_n,Q_n) = \sum_{y\in V_n\, z\in V: (y,z)\in E} \Phi\big( Q_n(y,z), Q^{\mu_n}(y,z)\big).
  \end{equation}

  We claim that $\Phi\big( Q_n(y,z), Q^{\mu_n}(y,z)\big)=0$ if $(y,z)$
  is as in the above sum and $Q^{\mu_n}(y,z)=0$. Since $y\in V_n$ then
  $Q^\mu(y,z)= \mu(V_n) Q^{\mu_n}(y,z) = 0$.  As $I(\mu,Q)<+\infty$ it
  follows $Q(y,z)=0$ and therefore $Q_n(y,z)=0$, which concludes the
  proof of the claim.  As a consequence, we can restrict the sum in
  \eqref{sum} to $Q^{\mu_n}(y,z)>0$

  Recall the definition of $\Phi $ given in \eqref{Phi}. Given $0\le
  q'\le q$ and $p'\ge p> 0$, let $\alpha,\beta\ge 0$ be respectively
  defined by $q'=q(1-\alpha)$ and $p'=p(1+\beta)$. Then we have
  \begin{equation}
    \label{acqua74}
    \begin{split}
      &
      \Phi(q',p') - \Phi(q,p)
      = q' \Big( \log\frac{q'}{p'} - \log\frac{q}{p}\Big)
      + (q'-q) \log\frac qp + (q-q')+ (p'-p)
      \\
      &\qquad
      \le (q'-q) \log\frac qp + (q-q')+ (p'-p)
      = - \alpha \,\Phi(q,p) +(\alpha+\beta) \, p
      \le (\alpha+\beta) \, p.
    \end{split}
  \end{equation}
  By construction, it holds $\mu_n(y) \ge \mu(y)$ for $y\in V_n$
  and $Q_n(y,z)\le Q(y,z)$ for $(y,z)\in E_n$.
  We set $\beta_n:= [\mu(V_n)]^{-1} -1$ and $\alpha_n\colon E_n
  \to [0,1]$ be defined by $Q_n(y,z) =Q(y,z)
  \big[1-\alpha_n(y,z)\big]$ when $(y,z)\in E(Q)$.
  From \eqref{acqua74} we then obtain
  \begin{equation*}
    I(\mu_n,Q_n)
    \le I(\mu,Q)
    +  \sum_{y\in V_n\, z\in V: (y,z)\in E} \big[\beta_n+
    \alpha_n(y,z)\big]\,\mu(y)r(y,z).
  \end{equation*}
  Since $I(\mu,Q)<+\infty$ then it
  holds $\langle\mu,r\rangle <+\infty$.    Since
  $\beta_n, \alpha_n(y,z)\downarrow 0$ and the maps $\alpha_n(\cdot) $
  are uniformly bounded, by dominated convergence we conclude the
  proof of \eqref{dt}.
\end{proof}


\section{Direct proof of Theorem
\ref{LDP:misura+flusso}}\label{s:diretto}

In this section we give a direct proof of Theorem
\ref{LDP:misura+flusso}, independent from the LDP for the empirical
process. As already mentioned, the proof works only  under the
additional condition that the graph $(V,E)$ is \emph{locally finite}
(cf.\ Condition \ref{t:ccompls}--(iii)).
This assumption implies that, given $\phi \in C_0(V)$, the function
$\nabla \phi: E \to \bb R$ defined as $\nabla \phi (y,z)=
\phi(y)-\phi(z)$ belongs to $C_0(E)$. As a consequence, the map
 \begin{equation}\label{fanta}L^1_+ (E)\ni Q \to \langle\phi,\div Q\rangle=-
 \langle  \nabla \phi, Q \rangle \in \bb R
 \end{equation} is continuous. Since a
 linear functional on $L^1_+ (E)$ is  continuous w.r.t. the bounded
 weak* topology if and only if it is continuous w.r.t. the weak*
 topology \cite{Me}, by definition of weak* topology the map defined in
 \eqref{fanta} is continuous (w.r.t. the bounded weak* topology)
 if and only if $\nabla \phi \in C_0(E)$. Hence, our additional condition is equivalent
 to the fact that  \eqref{fanta}
 is  continuous for any $\phi \in C_0(V)$.
 An explicit example of a not locally finite graph
 where \eqref{fanta} becomes not continuous for $\phi= \id_x $, $x
 \in V$, is given in Appendix \ref{div_disco}.

\subsection{Upper bound}
\label{s:ub}

Given $\phi\in C_0(V)$ and $F\in C_c(E)$ (i.e.\  $\phi$ vanishes at
infinity and $F$ is nonzero only on a finite set) let
$I_{\phi,F}\colon \mc P(V)\times L^1_+(E)\to \bb R$ be the map
defined by
\begin{equation}
  \label{Iff}
  I_{\phi,F} (\mu,Q) :=
  \langle\phi,\div Q\rangle + \langle Q, F\rangle - \langle\mu,r^F-r\rangle
\end{equation}
where $r^F\colon V \to (0,+\infty)$ is defined by $r^F(y)
=\sum_{z\in V} r(y,z)e^{F(y,z)}$ and $\langle\phi,\div
Q\rangle=\sum_{y\in V} \phi(y)\div Q(y) $.

\begin{lemma}
  \label{t:pub}
  Fix $x\in V$. For each $\phi\in C_0(V)$, $F\in C_c(E)$, and each
  measurable $\mc B\subset \mc P(V)\times L^1_+(E)$, it holds
  \begin{equation*}
    \varlimsup_{T\to+\infty}\; \frac 1T
    \log \bb  P_x \Big( (\mu_T,Q_T) \in \mc B \Big)
    \le -\inf_{(\mu,Q)\in \mc B}  I_{\phi,F} (\mu,Q).
  \end{equation*}
\end{lemma}

\begin{proof}
  Fix $x\in V$ and observe that the following pathways continuity
  equation holds $\bb P_x$ a.s.\
  \begin{equation}
    \label{pce}
    \delta_y(X_T)-\delta_y(X_0) + T \div Q_T(X) (y) =0
    \qquad \forall\: y\in V.
  \end{equation}

  Fix $F\in C_c(E)$ and $\phi\in C_0(V)$ and recall the semimartingale $\bb M^F$ introduced in Lemma~\ref{t:em2}. In view of
  \eqref{Iff} and \eqref{pce}, for each $T>0$ and each measurable set
  $\mc B\subset \mc P(V)\times L^1_+(E)$
  \begin{equation*}
    \begin{split}
      &\bb P_x \big( (\mu_T,Q_T) \in \mc B \big)
      \\
      &\qquad
      = \bb E_x \Big(
      \exp\big\{ -T \, I_{\phi,F} (\mu_T,Q_T) -
      \big[ \phi(X_T)-\phi(x) \big] \big\}
      \: \bb M_T^F \: \id_{\mc B}(\mu_T,Q_T) \Big)
      \\
      &\qquad \le
      \sup_{(\mu,Q)\in\mc B} e^{- T \, I_{\phi,F} (\mu,Q) }
      \; \bb E_x \Big(
      \exp\big\{- \big[\phi(X_T)-\phi(x)\big] \big\}
      \: \bb M_T^F \: \id_{\mc B}(\mu_T,Q_T) \Big).
    \end{split}
  \end{equation*}
  Since $\phi$ is bounded, the proof is now achieved by using
  Lemma~\ref{t:em2}.
\end{proof}

We can conclude the proof of the upper bound in
Theorem~\ref{LDP:misura+flusso}.  In view of the exponential tightness
proven in Subsection~\ref{s:et}, it is enough to prove \eqref{ubldp}
for compacts.  Since the graph $(V,E)$ is locally finite the map
$I_{\phi,F}$ is continuous.
Fix $x\in V$. By Lemma~\ref{t:pub} and the min-max lemma in
\cite[App.~2, Lemma~3.3]{KL} for each compact $\mc K\subset \mc
P(V)\times L^1_+(E)$ it holds
  \begin{equation*}
    \varlimsup_{T\to+\infty}\;
    \frac 1T \log \bb  P_x \Big( (\mu_T,Q_T) \in \mc K \Big)
    \le -\inf_{(\mu,Q)\in \mc K} \; \sup_{\phi,F} \; I_{\phi,F}(\mu,Q)
  \end{equation*}
  where the supremum is carried out over all $\phi\in C_0(V)$ and
  $F\in C_c(E)$. Recalling \eqref{rfq}, it is now simple to check (see Appendix \ref{iobimbo}) that
  for each $(\mu,Q)\in \mc P(V)\times L^1_+(E)$ it holds
  \begin{equation}
    \label{vci}
    I(\mu,Q) = \sup_{\phi,F} \; I_{\phi,F} (\mu,Q)\,,
  \end{equation}
  which concludes the proof of the upper bound.


\subsection{Lower bound}
\label{s:lb}

Recall the following general result concerning the large deviation
lower bound.

\begin{lemma}
  Let $\{P_n\}$ be a sequence of probability measures on a completely
  regular topological space $\mc X$. Fix $J\colon \mc X\to
  [0,+\infty]$ and assume that for each $x\in\mc X$
  there exists a sequence of probability measures $\{\tilde{P}_n^x\}$
  weakly convergent to $\delta_x$ and such that
  \begin{equation}
    \label{entb}
    \varlimsup_{n\to\infty} \frac 1n \Ent\big(\tilde{P}_n^x \big| P_n\big)
    \le J(x).
  \end{equation}
  Then the sequence $\{P_n\}$ satisfies the large deviation lower
  bound with rate function given by $\sce J$, the lower semicontinuous
  envelope of $J$, i.e.
  \begin{equation*}
    (\sce J) \, (x) := \sup_{U \in\mc N_x} \; \inf_{y\in U} \; J(y)
  \end{equation*}
  where $\mc N_x$ denotes the collection of the open neighborhoods of $x$.
\end{lemma}

This lemma has been originally proven in \cite[Prop.~4.1]{Je} in a
Polish space setting. The proof given in \cite[Prop.~1.2.4]{Ma}
applies also to the present setting of a completely regular
topological space.

Recall the definition of the set $\mc S$ given before Proposition
\ref{t:dt}:  $\mc S$ is given by the  elements $(\mu,Q)\in \mc
P(V)\times L^1_+(E)$ with $I(\mu,Q)<+\infty$ and  such that the
graph $(\supp(\mu), E(Q) )$ is finite   and connected.

 First we prove the entropy bound \eqref{entb} with $J$
given by the restriction of $I$, as defined in \eqref{rfq}, to $\mc
S$, that is
\begin{equation}
  \label{Jrf}
  J(\mu,Q):=
  \begin{cases}
    I(\mu,Q)  & \textrm{ if } (\mu,Q)\in\mc S\\
    +\infty   & \textrm{ otherwise. }
  \end{cases}
\end{equation}
Then we complete the proof of the lower bound \eqref{lbldp} by
showing that the lower semicontinuous envelope of $J$ coincides with
$I$.

\begin{lemma}
  \label{t:plb}
  Fix $x\in V$ and set $P_T := \bb P_x \circ (\mu_T,Q_T)^{-1}$.  For
  each $(\mu,Q)\in \mc P(V)\times L^1_+(E)$ there exists a sequence
  $\{\tilde{P}^{(\mu,Q)}_T\}$ of probability measures on $\mc P(V)\times
  L^1_+(E)$ weakly convergent to $\delta_{(\mu,Q)}$ and such that
  \begin{equation*}
    \varlimsup_{T\to+\infty} \frac 1T
    \Ent\big(\tilde{P}^{(\mu,Q)}_T \big| P_T \big)
    \le J(\mu,Q).
  \end{equation*}
\end{lemma}

\begin{proof}
  By definition \eqref{Jrf} of $J$, we can restrict to $(\mu,Q)\in \mc S$.
  First we discuss the case when  $x\in K:=\supp(\mu)$.
  We denote by $\tilde{\bb P}_x^{(\mu,Q)}$ the distribution of the
  Markov chain $\tilde{\xi}^x$ on $V$ starting from $x$ and having jump rates
  \begin{equation}
    \label{tilderat}
    \tilde{r} (y,z) :=
    \begin{cases}
      \frac{Q(y,z)}{\mu(y)} & \textrm{ if } (y,z)\in E(Q)\\
      \; 0 & \textrm{ otherwise.}
    \end{cases}
  \end{equation}
  Observe that this perturbed chain can be thought of as an irreducible  chain on the
  finite state space $K$. Moreover, the
  condition $\div Q =0$ implies that $\mu$ is the invariant probability
  measure.

  Set $\tilde{P}^{(\mu,Q)}_T := \tilde{\bb P}_x^{(\mu,Q)} \circ
  (\mu_T,Q_T)^{-1}$. The ergodic theorem for finite state Markov
  chains and the law of large numbers for the empirical flow discussed
  in Section~\ref{s:emef} imply that $\{\tilde{P}^{(\mu,Q)}_T\}$
  converges weakly to $\delta_{(\mu,Q)}$.
  We observe that
  \begin{equation}
    \begin{split}
    & \frac 1T  \Ent\big( \tilde{P}^{(\mu,Q)}_T \big| P_T \big)
    \le \frac 1T
    \Ent\Big( \tilde{\bb P}^{(\mu,Q)}_x\big|_{[0,T]} \: \Big| \:
    \bb P_x \big|_{[0,T]} \Big)
    \\
    &\quad
    =  \sum_{y\in K\,,\, z:(y,z)\in E}
    \tilde{\bb E}^{(\mu,Q)}_x
    \Big(
     Q_T(y,z) \log \frac{Q(y,z)}{\mu(y) r(y,z)}
    - \mu_T(y) \Big[ \frac{Q(y,z)}{\mu(y)} - r(y,z) \Big] \Big)
    \end{split}
    \label{troppilabel}
  \end{equation}
 where the subscript $[0,T]$ denotes the restriction to the
interval $[0,T]$ (above we used the convention $0\log 0:=0$).
Indeed, the first inequality follows from the variational characterization of
the relative entropy (see \cite[Sec.~2]{DV4}--(IV)) and the second
from a straightforward computation of the Radon-Nikodym density
(recall \eqref{RN}).
Since
$
T \tilde{\bb E}^{(\mu,Q)}_x
    \bigl(
     Q_T(y,z) \bigr)=\tilde{\bb E}^{(\mu,Q)}_x
    \bigl( \langle \mu_T, \tilde{r} \rangle \bigr)
 $
(adapt   \eqref{dacitare} to the present setting) and since $\mu_T (y) \to \mu (y)$  $\tilde{\bb P}^{(\mu,Q)}_x$--a.s. by  ergodicity,
 the
r.h.s. of \eqref{troppilabel} converges in the limit $T\to +\infty$
to
$$
\sum_{y,z\in K: (y,z)\in E }\Big(
     Q(y,z) \log \frac{Q(y,z)}{\mu(y) r(y,z)}
    + \mu(y)r(y,z)-Q(y,z)\Big)+\sum_{y\in K}\mu(y)\sum_{z\not\in
    K}r( y,z)\,,
$$
that is $I(\mu,Q)$.

When $x\not \in K$ then there exists an oriented path on $(V,E)$
from $x$ to $K$ since $(V,E)$ is connected. In this case the
perturbed Markov chain $\tilde{\xi}^x$ is defined with rates
\eqref{tilderat} with exception that $\tilde r(y,z):=r(y,z)$ for any
$(y,z)$ belonging to the  oriented path from $x$ to $K$ (fixed once
for all).
Since after a finite number of jumps that Markov chain reach the
component $K$, it is easy  conclude the proof by  the same
computations as before.
\end{proof}

Recall \eqref{rfq} and \eqref{Jrf}. Since
 $I$ is lower
semicontinuous and convex on $\mc P(V)\times L^1_+(E)$ (see Remark \ref{silente}), the inequality $\sce J \ge I$ holds. The proof of the
equality $I= \sce J$ is therefore completed by Proposition
\ref{t:dt}.

\section{Projection from the empirical process: proof of Theorems \ref{LDP:misura+flusso},  \ref{LDPteo2}}
\label{s:proiezione}

We recall the definition of the \emph{empirical process} referring to
\cite{DV4}--(IV), \cite{Vld} for more details.  We consider the space
$ D(\bb R; V)$ endowed of the Skorohod topology and write $X$ for a
generic element of $D(\bb R;V)$.  Given $X \in D(\bb R_+;V)$ and
$t>0$, $X^t\in D(\bb R;V)$ is the $t$--periodic path which
coincides with $X$ on $[0,t)$, that is
$$
\begin{cases}
X^t_s:=  X_s \text{ for } 0\leq s <t\,,\\
X^t_{s+t}:= X^t_s \text{ for } s \in \bb R\,.
\end{cases}
$$
Writing $\mc M_S $ for the space of stationary probabilities on $D(\bb R
;V)$  endowed of the weak topology, given $X\in D(\bb R_+;V)$ and
$t>0$ we denote by $\mc R_{t,X}$ the element in $\mc M_S $ such that
$$\mc R_{t,X} (A)= \frac{1}{t} \int _0^t \chi _A\bigl( \theta _s
X^t\bigr) ds\,, \qquad \forall A \subset D(\bb R;V) \text{ Borel}\,,
$$
where $ (\theta_s X^t)_u:= X^t _{s+u}$.
  Since $X \to \mc R_{t,X}$ is a Borel map from
$D(\bb R_+;V)$ to $\mc M_S $, for each $x \in V$ it induces a
probability measure $\Gamma_{t,x}$ on $\mc M_S $ defined as
$\Gamma_{t,x} := \bb P_x \circ \mc R_{t,X}^{-1}$. The above
distribution $\Gamma_{t,x}$ corresponds to the $t$--periodized
empirical process.

\smallskip

Let us denote by $\bar R$ the stationary process in $\mc M_S$
associated to the Markov chain $\xi$ and having  $\pi$ as marginal
distribution. By the ergodic theorem \eqref{ergodico},
$\Gamma_{t,x}$ weakly converges to $\delta_{\bar R}$ as $t \to
+\infty$, for each $x \in V$. As proven in  \cite{DV4}--(IV), under
the Donsker--Varadhan  condition,   for each $x \in V$ as $t\to
+\infty $ the family of probability measures $\Gamma_{t,x}$
satisfies a LDP with rate $t$ and rate function given by  the
relative entropy per unit of time $H$ w.r.t. the Markov chain $\xi$.

\smallskip

 We briefly
recall the  definition of $H$ and some of its properties, referring
to \cite{DV4}--(IV) for more details.
Given $-\infty \leq s \leq t \leq \infty$, let $\mc F ^s_t$ be the
$\sigma$--algebra in $D(\bb R;V) $ generated by the functions
$(X_r)_{s\leq r \leq t}$. Let $R\in \mathcal M_S$ and $R_{0,X}$ be the
regular conditional probability distribution 
of $R$ given $\mc F^{-\infty}_0$, evaluated on the path $X$. Then
$H(R) \in [0, \infty]$ is the only constant such that $ H(t, R)= t
H(R)$ for all $t>0$, where
 \begin{equation}\label{treporcellini} H(t, R):= \bb E_{R} \left[ H_{\mc F^0_t}\bigl(
 R_{ 0, X} \,\big| \, \bb P_{X_0}\bigr) \right]\,,
 \end{equation}
 $ H_{\mc F^0_t}\bigl(
 R_{ 0, X} \,\big| \, \bb P_{X_0}\bigr)$
  being  the relative entropy  of  $R_{ 0, X}$
 w.r.t. $\bb P_{X_0}$ thought of as probability measures on the
 measure space $D(\bb R;V)$ with measurable sets varying in the
 $\sigma$--subalgebra $\mc F^0_t$. The entropy $H(R)$ can be also
 characterized as the limit
   $
 H( R)= \lim _{t \to \infty} \bar{H} (t, R) /t$,
where \begin{equation}\label{babelino} \bar{H} (t, R):= \sup
_{\varphi \in \mc B (\mc F^0_t)} \left[ \bb E_R ( \varphi) - \bb
E_R\bigl( \log \bb E _{X_0 } (e^\varphi) \bigr) \right]
\end{equation}
and $\mc B ( \mc F ^0_t)$ denotes the family of bounded $\mc
F^0_t$--measurable  functions on $D( \bb R; V)$. Below we will
frequently use that \begin{equation}\label{vedoluce} t H(R) =
H(t,R)\geq \bar H (t,R) = \sup _{\varphi \in Y_1(t)} \bb E_{R} (
\varphi)\,,
\end{equation}
 where $Y_1(t)$ is the
family of functions $ \varphi \in \mc B( \mc F^0_t)$ such that $\bb
E_x \bigl( e^\varphi\bigr) \leq 1$ for all $x \in V$ (the last
identity is an immediate restatement of \eqref{babelino}).

In the following proposition we investigate some key identities
concerning the map $R \to \bigl(\hat\mu (R),\hat Q (R)\bigr)$. Recall
the definitions of $\hat \mu (R)$ and $\hat Q(R)$ given before Lemma
\ref{pasquetta}.

\begin{proposition}\label{cervo1pezzo}
Assume the Markov chain satisfies (A1)--(A4).
Then $\hat \mu \bigl( \mc R _{T,X}\bigr) = \mu_T (X)$ and
 $\hat Q \bigl( \mc R_{T,X} \bigr)=
 Q_T(X^T)
 \in L^1_+ (E)$ for
$\bb P_x$--a.e.\ $X \in D(\bb R_+; V)$.
\end{proposition}

\begin{proof}
 The fact that   $\hat \mu
\bigl( \mc R_{T,X}\bigr) = \mu_T (X)$ $\bb P_x$--a.s. has already
been observed in \cite{DV4}--(IV). Let us prove that $\hat Q \bigl(
\mc R_{T,X} \bigr) = Q_T(X^T)$ $\bb P_x$--a.s. It is convenient to
introduce the following notation: given $(y,z)\in E$, $X\in D(\bb
R_+; V)$ and     $I \subset \bb R_+$ we write $N_{I}(X) (y,z)$ for
the number of jumps along $(y,z)$ performed by the path $X$ at some
time in  $I$. In addition we write $N_T (X) (y,z)$ for
$N_{[0,T]}(X) (y,z)$.
 Equivalently, $N_T(X) (y,z)=T Q_T(X)
(y,z)$. Given $T>0$, fix $a\in (0,T)$. We then have
\begin{equation*}
\begin{split}
\hat Q\bigl( \mc R_{T,X} \bigr)  (y,z)   &  = \frac{1}{a} \bb E
_{\mc R_{T,X} } \left( \, N_a (y,z)\, \right)= \frac{1}{a T}\int_0^T
N_a\bigl( \theta_s X^T\bigr)  (y,z)  ds \\& =
 \frac{1}{a T}\int_0^T N_{[s,s+a]}(X^T) (y,z) ds\,.
\end{split}
\end{equation*}
 Let us write $0\leq t_1< t_2 < \cdots < t_n \leq  T$
for the times in $[0,T]$ at which the path $X^T$ jumps from $y$ to
$z$. Note that $n= N_T(X^T)(y,z)$. We denote by $\pi_T : \bb R \to
\bb R/ T \bb Z$ the canonical projection of $\bb R$ on the circle of
length $T$. It maps bijectively $[0,T)$ on $\bb R/ T \bb Z$.
Moreover, we define the set  $\Theta_T(X^T) (y,z):= \{\pi_T(t_1),
\pi_T(t_2) , \dots , \pi_T(t_n) \}$. Since $T>a$ the number
$N_{[s,s+a]}(X^T)(y,z) $ of jumps from $y$ to $z$ made by $X^T$ in
the time interval $[s,s+a]$ coincides with the cardinality of
$\Theta_T(X^T)(y,z)\cap \pi_T( [s,s+a]) $. Hence
\begin{multline}\label{rio3D}
\hat Q
\bigl( \mc R_{T,X} \bigr) (y,z)  = \frac{1}{a T} \int _0^T
\left| \Theta_T(X^T) (y,z)
\cap \pi_T([s,s+a ]) \right|\, ds =\\
 \sum _{k=1} ^n \frac{1}{a T}
\int _0^T \id \left( \pi_T(y_k) \in \pi_T([s,s+a]) \right) ds= \sum
_{k=1} ^n \frac{1}{T}= Q_T(X^T) (y,z)\,.
\end{multline}
\end{proof}

Note that, since  $\bb P_x$--a.s. time $T$ is not a jump time, it
holds
\begin{equation}\label{spostare}
Q_T(X^T) (y,z)
=
\begin{cases} Q_T(X) (y,z)
+\frac 1T& \text{ if }
(X_{T-},X_0)=(y,z) \in E\,,\\
Q_T(X) (y,z)
& \text{ otherwise}\,,
\end{cases}
\qquad \bb P_x\text{--a.s.}
\end{equation}

\medskip

In what follows, in order to allow a better overview of the proof of
Theorems \ref{LDP:misura+flusso} and \ref{LDPteo2}, we focus on the
main steps, postponing some technical details in subsequent
sections. We start with Theorem \ref{LDPteo2}, since the product
topology on the flow space is simpler.

\subsection{Proof of Theorem \ref{LDPteo2} }\label{cassandro}
The proof is based on  the generalized contraction principle
related to   the concept of exponential approximation discussed in
\cite[Sec.~4.2.2]{DZ}. To this aim,
 given $\epsilon \in (0,1/2)$, we
fix a  continuous function $\varphi_\epsilon: \bb R\to [0,1]$ such
that $\varphi _\epsilon (x)=0$ if $ x \not \in (0,1)$ and $ \varphi
_\epsilon (x)=1$ if $x \in [\epsilon, 1-\epsilon]$.  For each $(y,z)
\in E$   we consider the continuous and bounded function $F^\epsilon
_{y,z}: D(\bb R;V) \to \bb R$ defined as
\begin{equation*}
F_{y,z}^\epsilon (X):= \Big\{ \sum _{s \in [0,1]} \varphi_\epsilon
(s) \id\bigl( X_{s-}=y\,,\; X_s =z \bigr)\Big \} \wedge \epsilon
^{-1}.
\end{equation*}
Then, we define $\hat Q _\epsilon  : \mc M_S \to [0, +\infty]^E$ as
$\hat Q_\epsilon (R) (y,z) := \bb E_R \bigl (F _{y,z}^\epsilon
\bigr)$. Note that $\hat Q _\epsilon$ maps $\mc M_S$ into $[0,
\epsilon ^{-1}]^E$.

\begin{proposition}\label{cervo2} Assume the Markov chain satisfies
(A1)--(A4). Consider the space $[0,+\infty]^E$ endowed of the
product topology and the Borel $\sigma$--algebra. Then the following
holds:

\begin{itemize}
\item[(i)] The map $(\hat\mu, \hat Q):  \mc M_S\to \mc P(V)
\times [0,+\infty] ^E$ is measurable and the map $\hat \mu: \mc M_S
\to \mc P(V)$ is continuous.

\item[(ii)] The maps   $\hat
Q_\epsilon: \mc M_S \to [0,+\infty]^E$, parameterized by $\epsilon\in
(0,1/2)$, are continuous and satisfy
\begin{align}
& \lim _{\epsilon \downarrow 0} \sup _{R \in \mc M_S \,:\, H(R) \leq
\alpha} \bigl| \hat Q (R) \, (y,z) - \hat Q_\epsilon (R)\, (y,z)\bigr|=0
\,, \label{dz1}
\\
& \lim_{\epsilon  \downarrow 0} \varlimsup_{T \uparrow \infty}
\frac{1}{T} \log  \Gamma_{T,x} \left( \bigl| \hat Q  (y,z)  - \hat
Q_\epsilon (y,z) \bigr|
> \delta \right)=-\infty\,,\label{dz2}
\end{align}
for any $x \in V$, $\alpha>0$, $\delta>0$ and any edge $(y,z) \in E$.
\end{itemize}
\end{proposition}

As shown below, if $H(R)<+\infty$ then $\hat Q(R) \in \bb R_+^E$.  In
addition $\hat Q_\epsilon$ always assumes finite values. In
particular, the quantities appearing in \eqref{dz1} and \eqref{dz2}
are finite and the subtraction is meaningful.  We postpone the proof
of Proposition \ref{cervo2} to Section \ref{dim_cervo2} and conclude
the proof of Theorem \ref{LDPteo2}.

To prove item (i)  up to \eqref{samarcanda1}  we apply Theorem
4.2.23 in \cite{DZ}. Identity \eqref{dz1} corresponds to formula
(4.2.24) there, while identity \eqref{dz2} states, following the
terminology in \cite{DZ}, that the family of
probability measures $\left\{\Gamma_{T,x}\circ \bigl(\hat\mu,\hat
Q_\epsilon)^{-1}\right\}$ is an exponentially good approximation of
the family   $\left\{\Gamma_{T,x}\circ \bigl(\hat\mu,\hat
Q)^{-1}\right\}$. Combining the last observations with  the LDP of
the empirical process proved in \cite{DV4}--(IV), one gets the
thesis for the family  of probability measures $\{\bb
  P_x\circ (\mu_T,\tilde{Q}_T)^{-1}\}$ on $\mc P(V)\times \bb [0,+\infty]^E$
  where $ \tilde{Q}_T(X):= Q_T(X^T)$  (use Proposition
  \ref{cervo1pezzo}). At this point, due to Theorem 4.2.13 in
  \cite{DZ}, we only need to prove that the families of probability
  measures
$\{\bb
  P_x\circ (\mu_T,Q_T)^{-1}\}$ and $\{\bb
  P_x\circ (\mu_T,\tilde{Q}_T)^{-1}\}$ are exponentially
  equivalent. It is enough  to show that  for each $\delta>0$  it holds
  \begin{equation}\label{rotture}
  \varlimsup _{T \to +\infty} \frac{1}{T} \log \bb P_x \bigl( D(
  \tilde{Q}_T, Q_T) >\delta)=-\infty\,,
  \end{equation}
  where $D(\cdot, \cdot)$ denotes the metric of $[0,+\infty]^E$
  introduced at the beginning of
  Subsection \ref{prodotto}.
 By \eqref{spostare} $\tilde{Q}_T (y,z)= Q_T(y,z)$
with exception of at most one  edge $(y,z)$ where it holds $\tilde{Q}_T (y,z)= Q_T(y,z)+1/T$. Since
$|a/(1+a)-(a+\Delta)/(1+a+\Delta)|\leq \Delta$ for $a, \Delta \geq
0$, we conclude that  $D(
  \tilde{Q}_T, Q_T)\leq 1/T$, thus allowing to end the proof.

\subsection{Proof of \eqref{samarcanda}}

\subsubsection{Proof of \eqref{samarcanda} for $Q \not \in [0,+\infty)^E$}


Let $Q\in [0,+\infty]^E$ be such that $Q(y,z)=+\infty$ for some $(y,z)
\in E$. We need to show $\tilde I (\mu,Q) = +\infty$.
By Remark \ref{SD} (stochastic domination), it holds
$C:=\sup_{x \in V} \bb E _x \left( e^{Q_T (y,z) }\right) < +\infty.$
Hence for $ \lambda>0$ the function
$\varphi(X) := Q_T(X) (y,z) \id ( Q_T(X) (y,z) \leq\lambda)- \log C$
belongs to $Y_1(T)$.
By \eqref{vedoluce} we get
$$
 TH(R) \geq \bar H(T,R)\geq \bb E_R( \varphi)
$$
and we conclude by taking the limit $\lambda\to \infty$.


\subsubsection{Proof that $I(\mu,Q) \leq \tilde I(\mu,Q)$  for
$(\mu,Q)\in \mc P (V) \times L^1_+ (E)$}

$\,$
\smallskip

\noindent
Given $y \not= z$ in $V$  define $Q_T(X) (y,z)$  as the $T$ times the number of jumps up to time $T$  along $(y,z)$ in the trajectory $X$.

\begin{lemma}\label{befana}
If $R \in \mc M_S$ and $H(R)<+\infty$ then $R \bigl( Q_T(y,z) >0\bigr)=0$ for all $T\geq 0$ and $(y,z) \in (V \times V) \setminus E$ with $y \not = z$.
\end{lemma}
\begin{proof}
Take the  function $\varphi(X):= \lambda \id ( Q_T
(y,z)
>0) $ for fixed $\lambda >0$. Note that  $\varphi \in Y_1 (T)$
since  $\varphi\equiv 0$ $\bb P_x $--a.s.
Hence by \eqref{vedoluce} we  get
$$ TH(R) \geq \bar H(T,R)\geq \bb E_R( \varphi)= \lambda R\left(
Q_T(y,z)>0\right) \,.$$ Since $H(R)<+\infty$  the thesis follows by taking $\lambda$ arbitrarily
large.
\end{proof}

\begin{lemma}\label{l1}
Given $R \in \mc M _S$ with   $H(R)<+\infty$, it
holds
$$ \sum _{ z\,:\,(y,z)\in E} \hat Q(y,z)= \sum _{z\,:\,(z,y) \in E}
\hat Q(z,y)\,, \qquad \hat Q= \hat Q(R)\,.$$
\end{lemma}
\begin{proof}
The thesis follows by using Lemma \ref{befana} and considering  the $R$--expectation of the following
identity on $D([0,T];V)$:
$$ \id (X_T=y) + \sum_{z\,:\, z \not =y}T Q_T(X) (y,z)=
\id (X_0=y)+\sum_{z\,:\, z \not =y} T Q_T(X)  (z,y)\,.\qedhere
$$
\end{proof}

  Fix $(\mu,Q) \in \mc P (V) \times  L^1 _+(E)$. By Lemma \ref{l1}, if
$\div Q \not =0$  then  $\tilde I(\mu,Q)= +\infty=I(\mu,Q)$.
Hence, from now on we can restrict to $\div Q=0$. Fix $R\in \mc M_S$
such that $Q= \hat Q(R)$ and
 $\mu = \hat \mu (R)$ (the absence of such an $R$ would imply $\tilde I(\mu,Q)=+\infty$ and there would be nothing to prove).

 \medskip

 We first consider the case that there is some edge $(y,z) \in E$ with $Q(y,z)>0$ and $\mu(y)=0$. Trivially  in this case  $I(\mu,Q)=+\infty$. Let us prove that $\tilde I(\mu,Q)=+\infty$.
 To this aim, given $\epsilon>0$,  we define the function $F_\epsilon:E \to \bb R$ as $F_\epsilon (u,v)=\log\frac{  Q(y,z)}{ \epsilon r(y,z)} \mathds{1} \bigl(\, (u,v)= (y,z) \,\bigr)$.   Let
$e^{\varphi_{\epsilon}}:=\bb M_T^{F_{\epsilon}}$ be the
supermartingale introduced in Lemma \ref{t:em2}:
\begin{equation}\label{crema} \varphi_{ \epsilon} =T  Q_T(y,z)\log\frac{  Q(y,z)}{ \epsilon r(y,z)}-T \mu_T(y)  \left[ \frac{ Q(y,z)}{\epsilon} -r(y,z) \right]\,.
\end{equation}
We take $\epsilon $ small enough so that $ \log\frac{  Q(y,z)}{ \epsilon r(y,z)}>0$ and define  for $\ell>0$ the new function $\varphi_{\epsilon , \ell}$ as in the r.h.s. of \eqref{crema}  with $Q_T(y,z)$ replaced by
  $Q_T(y,z)\wedge \ell$.
  Then $\varphi_{\epsilon , \ell} \leq \varphi_{\epsilon}
$ and by Lemma \ref{t:em2} we conclude that $\varphi_{\epsilon , \ell}\in Y_1(T)$. Applying
\eqref{vedoluce} we conclude that
$$  H(R) \geq \bb E_R \bigl ( \varphi_{\epsilon , \ell} \bigr) /T = \bb E_R( Q_T(y,z) \wedge \ell)
\log\frac{  Q(y,z)}{ \epsilon r(y,z)}\,.$$
Taking first the limit $\ell \to +\infty$ and afterwards $\epsilon \to 0$, we get that  $H(R)=+\infty$, thus implying $\tilde I (\mu,Q) =+\infty$.
  \medskip

Due to the previous result, we restrict to the case that $\mu(y)>0$ if $Q(y,z)>0$, with $(y,z)\in E$.
Then we fix an invading sequence $E_n\nearrow E$ of finite subsets of $E$  and consider  the function
$F_{n }: E \to\bb R$ defined as
$$
r^{F_n}(y,z) = r(y,z)e^{F_{n}(y,z)} :=
\begin{cases}
\frac{Q(y,z)}{\mu  (y)}\,, & \text{ if } (y,z)\in E_n \,,
\\
r(y,z) & \text{ otherwise}\,.  \end{cases}
$$
with the convention that $0/0=0$.   Note that the above ratio is well defined since $\mu
(y)  >0$ if  $Q(y,z)>0$. Let
$e^{\varphi_{n}}:=\bb M_T^{F_{n}}$ be the
supermartingale introduced in Lemma \ref{t:em2}:
$$ \varphi_{n} =T \sum _{(y,z) \in E_n} \Big \{ Q_T(y,z)
\log \frac{ Q(y,z) }{\mu (y) r(y,z)}  -\mu_T(y) r(y,z) \bigl[\frac{ Q(y,z) }{\mu (y) r(y,z)}
-1 \bigr]\Big\}\,.
$$
Since $\varphi_{n } $ is unbounded, for $\ell>0$ we
consider the cut--off
$$ \varphi_{n, \ell} := \begin{cases}
\varphi_n & \text{  if  } |\varphi_n | \leq \ell \,, \\
\frac{\varphi_n}{|\varphi_n|}  \ell & \text{  if  } |\varphi_n | > \ell \,.
\end{cases}
$$
We stress that the sum in the definition of $\varphi_n$ is
finite. Since $|\varphi_{n,\ell}| \leq |\varphi_n| \in L^1(R) $
(recall that $Q= \hat Q(R) \in L^1_+(E)$), by the Dominated
Convergence Theorem it holds $ \lim_{\ell \to +\infty} \bb E_R\left(
  \varphi_{n,\ell}\right)= \bb E_R \left(\varphi_n\right)$.
Moreover, there exist positive
constants $A_n,B_n$ depending only on $n$ such that
$$
|\varphi_{n,\ell} | \leq | \varphi_n |
\leq A_n \sum _{(y,z) \in E_n} TQ_T(y,z) +B_n.
$$
By Remark \ref{SD},
this implies that $\log \bb E_{x} \left( e ^{\varphi_{n,\ell}} \right
)$ is bounded uniformly in $x\in V$. Therefore, by
dominated convergence and Lemma \ref{t:em2}, we conclude that
 \begin{equation*}
  \begin{split}
   \lim_{\ell \to +\infty} \bb E_R \log \bb E_{X_0} \left( e
  ^{\varphi_{n,\ell}} \right ) & = \lim_{\ell \to +\infty} \sum_{x \in
  V} \mu(x) \log \bb E_{x} \left( e ^{\varphi_{n,\ell}} \right )
\\
&=\sum
_{x\in V} \mu(x) \log \bb E_x \left( e ^{\varphi_{n}} \right ) \leq
0\,.
  \end{split}
\end{equation*}
As a consequence
$$ \lim _{\ell \to \infty} \left\{ \bb E_R\left(
    \varphi_{n,\ell}\right)- \bb E_R  \log \bb E_{X_0} \left( e
    ^{\varphi_{n,\ell}} \right ) \right\} \geq  E_R\left(
  \varphi_{n}\right)\,.
$$
Combining the above estimate,  \eqref{babelino} and \eqref{vedoluce},  we conclude that
\begin{equation}\label{fabio}
  H(R)  \geq  \bar H(T,R)/T \geq  E_R\left( \varphi_{n}\right) /T =\sum
_{(y,z) \in E_n}\Phi( Q(y,z), Q^\mu(y,z) ) \,.
\end{equation}
To conclude we take the  limit $n \to +\infty$, obtaining $H(R) \geq I(\mu,Q)$ for each $R\in \mc M_S$ such that $\hat \mu(R)=\mu$, $\hat Q(R)= Q$. This implies that $\tilde I(\mu,Q) \geq I(\mu,Q)$.

\subsubsection{Proof that $I(\mu,Q) \geq \tilde I(\mu,Q)$  for
$(\mu,Q)\in \mc P (V) \times L^1_+ (E)$}\label{s:bivio}  As a
consequence of the first part of Theorem \ref{LDPteo2} (already
proved), the function $\tilde I$ is lower semicontinuous. Consider
the sequence  $\{(\mu_n, Q_n)\}_{n \geq 0} $ in $\mc S$ converging
to $(\mu,Q)$ as stated in Proposition \ref{t:dt}. The set $\mc S$
has been defined in Section \ref{lattino} as   the subset of $ \mc
P(V)\times L^1_+(E)$ given by the elements $(\mu,Q)$ with
$I(\mu,Q)<+\infty$ and  such that the graph $(\supp(\mu), E(Q) )$ is
finite   and connected. For each $n$ we consider the continuous time
Markov chain $\xi^{(n)}$ on $V$ with jump rates $r_n(y,z)=
Q_n(y,z)/\mu_n (y)$ with the convention $0/0=0$. Since
$I(\mu_n,Q_n)<+\infty$ it cannot be $Q_n(y,z)>0$ and $\mu_n(y)=0$,
hence the above ratio is well defined.
 Since $\mu_n$ and $Q_n$ have finite support, the Markov chain
$\xi^{(n)}$ has finite effective state space.
In particular, explosion does not take place. The bound
$I(\mu_n,Q_n)<+\infty$ implies also that $\div Q_n=0$, hence we get
that $\mu_n$ is an invariant measure for $\xi^{(n)}$. We define $R_n$
as the stationary Markov chain $\xi^{(n)}$ with marginal $\mu_n$, then
$\hat Q(R_n)=Q_n$. By the Radon--Nykodim derivative \eqref{RN} and the
definition of the entropy $H(\cdot)$, we get that $\tilde I
(\mu_n,Q_n)\leq H(R_n )= I(\mu_n,Q_n)$. 
 Invoking the lower
semicontinuity of $\tilde I$ and Proposition \ref{t:dt}, we get the
thesis.

\subsection{Proof of \eqref{fatto!}}

 Let us take $(\mu,Q) $ with $\mu \in\mc P(V)$
and $Q\in \bb R_+^E \setminus L^1_+(E)$. We need to prove that
$\tilde I(\mu,Q)=+\infty$.  Let $R \in \mc M_S $ be such that $\hat
\mu (R)=\mu$ and $\hat Q(R)=Q$ (we assume $R$ exists, otherwise the
thesis is trivially true). We fix an invading sequence $V_n \nearrow
V$ of finite sets, define  $E_n:= \{ (y,z) \in E\,:\, y,z \in V_n\}$
and  $F_n(y,z):= \id( (y,z) \in E_n )$ for $(y,z)\in E$. Then we
know that $\bb E_x \Big( \exp\{\bb M_T^{F_n} \}\Big)\leq 1$ for all
$x \in V$, using the same notation of Lemma \ref{t:em2}. Again we
need to work with functions in $\mc B (\mc F ^0_T)$. To this aim,
given $\ell >0$ we define $\bb M_{T,\ell}^{F_n}$ as the
supermartingale $\bb M_T ^{F_n}$ except that the empirical flow
$Q_T(y,z) $ is replaced by $Q_T(y,z) \wedge \ell $ for all edges
$(y,z)$.
Then (note that $r^{F_n} \geq r$)  $ \bb M_{T,\ell}^{F_n}\in \mc B (\mc F ^0_T)$
and $\bb M_{T,\ell}^{F_n} \leq \bb M_T^{F_n}$, thus implying that
$\bb M_{T, \ell}^{F_n} \in Y_1(T)$. By \eqref{vedoluce} this implies
that
\begin{equation}\label{amicoluca}
  H(R)\geq \bar H(T,R)/T
\geq \varlimsup_{\ell \to \infty} \bb E_{R}\Big( \bb M_{T,
\ell}^{F_n}\Big)/T= \sum _{(y,z)\in E_n} Q(y,z) - \bb E_R( \mu_T(
r^{F_n}-r) ) \,.
\end{equation}
The conclusion then follows from the next result:
\begin{lemma}\label{schnell}
Assume Condition \ref{t:ccomp} 
 (where the
constants $\sigma,C$ are defined). Then for each $R \in \mc M_S$ it
holds
\begin{equation}\label{armadio}
 \|\hat Q (R) \| \leq
H(R)(1+e/\sigma)+ C \,e/\sigma \,.
\end{equation}
\end{lemma}
\begin{proof}
Let us first prove \eqref{armadio} knowing that  $H(R) \geq \bb
E_R\bigl(v(X_0)\bigr) $ (this will be proved later). We come back to \eqref{amicoluca} and take
first the limit $T \to +\infty$ and afterwards the limit $n \to
+\infty$. Since $F_n(y,z)= \id( (y,z) \in E_n )$, then
$0\leq r^{F_n}-r\leq e r$. By Fubini--Tonelli
and stationarity, $ \bb E_R( \mu_T(r)) = \bb E_R( r (X_0) )$.
We then conclude that
$$ \|\hat Q\|=\|Q\| = \lim _{n\to +\infty}\sum _{(y,z)\in E_n} Q(y,z)\leq
H(R)+e \bb E_R( r (X_0) )\,.
$$
By Condition \ref{t:ccomp}, $ \bb E_R\bigl(r (X_0) \bigr)\leq \bb
E_R\bigl(v(X_0) \bigr) /\sigma+C/\sigma$. Combining with $H(R) \geq \bb
E_R\bigl(v(X_0)\bigr) $ we get the thesis.

\smallskip

Let us now prove that $H(R) \geq \bb E_R\bigl(v(X_0)\bigr) $.  Since
both $H(R)$ and $ E_R\bigl(v(X_0)\bigr) $ are affine in $R$ (see
\cite{DV4}--(IV)) and since all stationary processes are convex
combinations of ergodic stationary processes, it is enough to prove
the claim for an ergodic $R \in \mc M _S$. Given $k, T >0$ and $W
\subset V$ we define $v^{(k)}:= v\wedge k$ and $\varphi(X):=\id
(X_0\in W)   \int _0^T v^{(k)}(X_s) ds$. Trivially, $\varphi \in \mc
B ( \mc F^0_T)$. Then, by the definition of $\bar H(T,R)$,  it holds
\begin{equation}\label{iris}
\begin{split}
T H(R) & \geq \bar H(T,R) \geq \bb E_R(\varphi)- \bb E_R\bigl( \log
\bb E_{X_0} (e^\varphi)\bigr) \\& \geq \bb E_R\Big ( \int _0^T
v^{(k)}(X_s) ds; X_0 \in W \Big) - \max _{x \in W} \log (C_x/c) \,.
\end{split}
\end{equation}
In the last inequality we have used Lemma \ref{t:letem} and the
inequality $v^{(k)} \leq v$. At this point, we   divide \eqref{iris}
by $T$. Since $R$ is ergodic, by Birkhoff ergodic theorem (note that
$v^{(k)} (X_0) \in L^1 (R)$ since $v^{(k)}$ is bounded) we know that
$$ \lim _{T \to \infty}  \frac{1}{T}\int _0^T v^{(k)}(X_s) ds =\bb
E_R\bigl( v^{(k)}(X_0)\bigr)\,, \qquad \text{$R$--a.s.}$$  Taking
the limit $T \to \infty$ and applying the Dominated Convergence
Theorem we conclude that
$$ H(R)\geq \bb E_R\bigl(
v^{(k)}(X_0)\bigr) \bb  R(X_0\in W)\,.
$$
At this point it is enough to take the limit $k \to \infty$ and
afterwards to take $W$ arbitrarily large and invading all $V$.
\end{proof}

\subsection{Proof of Theorem \ref{LDP:misura+flusso}}
The proof uses the results of \cite{ES}, where the notion of
exponentially good approximation and the contraction principle are
extended to the case of completely regular space as image space of
the projection. To this aim we recall some further properties of the
bounded weak* topology on $L^1_+(E)$.

We define $\mc A$ as the set of  sequences
$\mathfrak{a}=(a_n)_{n\geq 1} $ of functions in $ C_0(E)$ such that
$\|a_n \|_\infty\to 0$. Given $\mathfrak{a} \in \mc A$ we introduce
the pseudometric $d_\mathfrak{a}$ on $L^1_+(E)$ as
$$d_\mathfrak{a} (Q,Q'):= \sup _{n \geq 1 } \langle Q - Q'  ,a_n \rangle \,.
$$
Writing $B_\mathfrak{a}(Q,r):= \{ Q'\in L^1_+ (E) \,:\,
d_\mathfrak{a}(Q,Q') <r\}$, the family of sets $\{ B_\mathfrak{a}
(Q,r)\}$, with $\mathfrak{a} \in \mc A$, $Q\in L^1_+(E)$ and $r>0$,
forms a basis for $L^1_+(E)$. This follows from Def. 2.7.1 and Cor.
2.7.4 in \cite{Me}. In addition, the family $\mc D$ of pseudometrics
$\{d_\mathfrak{a} \,:\, \mathfrak{a}\in C_0(E)\}$ is separating,
i.e. given $Q\not =Q'$ in $L^1_+(E)$ there exists $\mathfrak{a} \in
\mc A$ such that $d_\mathfrak{a}(Q,Q')>0$. The above two properties
(basis and separating family of pseudometrics) make $L^1_+(E)$ a so
called \emph{gauge space}. Indeed, one can prove that the concepts
of completely regular space and gauge space are equivalent
\cite[Ch. IX]{Du}.

\smallskip
Due to the above observations on the gauge structure of $L^1_+(E)$
we are in the same settings of \cite{ES}.  In what follows we
restrict to the case $|V|=+\infty$, thus implying $|E|=+\infty$ due
to the irreducibility of the Markov chain $\xi$ (the finite case is
much simpler).
Fix an enumeration
$(e_n)_{n\geq 1}$ of $E$. Consider the maps $\hat Q, \hat Q_\epsilon
$ entering in Proposition \ref{cervo2}
 and define the
maps $\bar Q, \bar Q_\epsilon: \mc M_S \to L^1_+(E)$ by %
\begin{align*}
  \bar Q(R)&=\begin{cases} \hat Q(R) &\; \;\;\text{ if }\hat Q(R)\in
L^1_+(E)\,,\\
0 &\;\;\; \text{ otherwise}\,, \end{cases} \\
 \bar Q_\epsilon (R)(e_n)&=\begin{cases} \hat Q_\epsilon (R)(e_n)  &
\text{ if } n\leq \epsilon^{-1}\,,\\
0 & \text{ otherwise}\,. \end{cases}
\end{align*}

\begin{proposition}\label{cervo3}
Assume the Markov chain satisfies (A1)--(A4)  and Condition
\ref{t:ccomp}.
 Consider the space $L^1_+(E) $ endowed of the  bounded weak*
topology  and the Borel $\sigma$--algebra. Then the following holds:
\begin{itemize}
\item[(i)] The map $\bar  Q:  \mc M_S\to L^1_+(E)$  is measurable while the maps
 $\bar  Q_\epsilon:  \mc M_S\to L^1_+(E)$
 are continuous.
\item[(ii)] For each $\mathfrak{a} \in \mc A$
\begin{align}
& \lim _{\epsilon \downarrow 0} \sup _{R \in \mc M_S \,:\, H(R) \leq
\alpha} d_\mathfrak{a} \bigl( \bar Q (R), \bar Q_\epsilon  (R)
\bigr)=0 \,, \label{dz1bis}
\\
& \lim_{\epsilon  \downarrow 0} \varlimsup_{T \uparrow \infty}
\frac{1}{T} \log  \Gamma_{T,x} \left(  d_\mathfrak{a} ( \bar  Q  ,
\bar Q_\epsilon) > \delta \right) =-\infty\,,\label{dz2bis}
\end{align}
for any $x \in V$, $\alpha>0$, $\delta>0$.
\end{itemize}

\end{proposition}
The proof is given in Section \ref{dim_cervo3}.

\smallskip

As byproduct  of Proposition \ref{cervo3}, the extended contraction
principle in \cite{ES}, the LDP of the empirical process and Theorem
\ref{LDPteo2}--(ii) we can conclude the proof of Theorem
\ref{LDP:misura+flusso}. Let us be more precise. We apply Theorem
1.13 in \cite{ES}. Formula \eqref{dz1bis} corresponds to formula
(1.14) in \cite{ES}, while formula \eqref{dz2bis} means that the
family of probability measures $\left\{\Gamma_{T,x}\circ
\bigl(\hat\mu,\bar Q_\epsilon)^{-1}\right\}$ is a
$(d_\mathfrak{a})_{\mathfrak{a} \in \mc A}$--exponentially good
approximation of the family $\left\{\Gamma_{T,x}\circ
\bigl(\hat\mu,\bar Q)^{-1}\right\}$. On the other hand, we have that
$\bar Q= \hat Q \in L^1_+(E)$ $\Gamma_{T,x}$--a.s., while by
Proposition \ref{cervo1pezzo} the random variable $\hat Q$ sampled
according to $\Gamma_{T,x}$ has the same law of $\tilde
Q_T(X):=Q_T(X^T)$ with $X \in D(\bb R_+; V)$ sampled according to
$\bb P_x$. Hence, by Corollary 1.10 in \cite{ES}
 we only need to prove that the families of probability
  measures
$\{\bb
  P_x\circ (\mu_T,Q_T)^{-1}\}$ and $\{\bb
  P_x\circ (\mu_T,\tilde{Q}_T)^{-1}\}$ are $(d_\mathfrak{a})_{\mathfrak{a} \in
\mc A}$--exponentially
  equivalent on $\mc P(V) \times L^1_+(E)$. It is enough  to show   for each $\delta>0$ and $\mathfrak{a} \in \mc A$ that
  \begin{equation}\label{rotturebis}
  \varlimsup _{T \to +\infty} \frac{1}{T} \log \bb P_x \bigl( d_\mathfrak{a}(
  \tilde{Q}_T, Q_T) >\delta)=-\infty\,,
  \end{equation}
Since by \eqref{spostare} $d_\mathfrak{a}(
  \tilde{Q}_T, Q_T)\leq \|\mathfrak{a}\|_\infty/T$, we get the
  thesis.

\section{Exponential approximations: Proof of Proposition
\ref{cervo2}}\label{dim_cervo2} Item (i) is straightforward. We
concentrate on item (ii). Since $\mc M_S$ is endowed of the weak topology and since $F
_{y,z}^\epsilon$  is a continuous bounded function on $D(\bb R;V)$
we conclude that $\hat Q_\epsilon $ is continuous.


\subsection{Proof of \eqref{dz1} }

As already proved in the previous section (independently from the
content of Proposition \ref{cervo2}), $\tilde I(\mu,Q) =+\infty$ if
$Q \not \in [0,+\infty)^E$. Hence, given $R \in \mc M_S$ with $H(R)<
+\infty$, it must be $\hat Q(R) (y,z)  < \infty$ for all $(y,z) \in
E$.
Below $R\in  \mc M_S$ is such that  $H(R)\leq \alpha$.

 Recall the
definition of $N_I(y,z)$ and  $N_T(y,z)$
given in the proof of Proposition \ref{cervo1pezzo}.
We can estimate
\begin{multline}\label{piangemolto}
\bigl| \hat Q (R) (y,z) - \hat Q_\epsilon(R) (y,z) \bigr|\\
\leq \bb E_R\bigl( N_1(y,z) ; N_1(y,z) \geq \epsilon^{-1} \bigr)+
\bb E_R\bigl( N_{ [0,\epsilon]\cup [1-\epsilon, 1]}(y,z) \bigr).
\end{multline}
By stationarity (see the proof of Lemma \ref{pasquetta}) $$\bb
E_R\bigl( N_{[0,\epsilon]} (y,z) \bigr) =\bb E_R\bigl(
N_{[1-\epsilon,1]} (y,z) \bigr)= \epsilon \bb E_R\bigl( N_1
(y,z)\bigr)= \epsilon  \hat Q (R) (y,z)\,.
$$
Consider $\ell\in \bb R_+$ and apply \eqref{vedoluce} with $t=1$ and $\varphi=N_1(y,z)\wedge \ell-r(y,z)(e-1)$ (note that $\varphi \in Y_1(t)$ by Remark \ref{SD}).
We get for $R\in \mathcal M_S$ such that $H(R)\leq \alpha$
\begin{equation}
\alpha+r(y,z)(e-1)\geq H(R)+r(y,z)(e-1)\geq
\mathbb E_R\left(N_1(y,z)\wedge \ell\right)\,.
\label{elalafesta}
\end{equation}
Since by the  monotone convergence  $\lim_{\ell\to +\infty}\mathbb E_R\left(N_1(y,z)\wedge \ell\right)=\hat Q (R) (y,z)$,
taking the limit $\ell\to +\infty$ on both extreme sides of \eqref{elalafesta} we deduce
$$
\alpha+r(y,z)(e-1)\geq
\hat Q (R)(y,z)\,.
$$
From this inequality we get
that the last term in \eqref{piangemolto} converges uniformly to
zero on $\left\{R\in \mathcal M_S\,: H(R)\leq \alpha\right\}$  as $\epsilon
\downarrow 0$.
To conclude, it remains to prove that
$\lim _{\epsilon \downarrow 0}  \bb E_R\bigl( N_1(y,z) ; N_1(y,z)
\geq \epsilon^{-1} \bigr) =0$.
 To this aim,  given $\gamma,\ell >0$  we define on $D(\bb R;
 V\bigr)$ the function
$$ \varphi_{\gamma , \ell, \epsilon}:= \gamma  N_1(y,z) \id  ( \ell \geq  N_1(y,z)  \geq \epsilon^{-1}) - C(\gamma,\epsilon)$$ where
$C(\gamma, \epsilon) := \sup_{x \in V} \log \bb E_x \bigl( e^{
\gamma N_1(y,z)\, \id \bigl(  N_1(y,z) \geq \epsilon^{-1}\bigr)
}\bigr )$. Due to Remark \ref{SD}
we get
 $ C(\gamma, \epsilon)<+\infty$ and $\lim _{\epsilon \downarrow 0}
 C(\gamma, \epsilon)=0$.
 By construction  $\varphi _{\gamma , \ell, \epsilon} \in Y_1(t)$ for $t\geq 1$. Applying
\eqref{vedoluce} we get for $t \geq 1$ that
$$ \bb E_R (\varphi_{\gamma , \ell, \epsilon}) \leq \bar H (t, R) \leq t H(R)\leq  t \alpha
\,.
$$
Taking $\ell \to \infty$, we conclude that
 $\bb E_R \bigl(  N_1(y,z) \,;\,  N_1(y,z)    \geq \epsilon^{-1} \bigr)
\leq t\alpha /\gamma + C(\gamma, \epsilon)/\gamma$. Taking first the
limit $\epsilon\downarrow 0$ and afterwards the limit $\gamma
\uparrow \infty$, we conclude  that the expectation  $ \bb E_R\bigl(
N_1(y,z) ; N_1(y,z) \geq \epsilon^{-1} \bigr)$ is negligible as
$\epsilon \downarrow 0$.\qed

\subsection{Proof of \eqref{dz2} }

We  restrict to $T>1$ (the generic case could be treated by the same
arguments of the proof of Proposition  \ref{cervo1pezzo}). Recall
the definition of the projection $\pi_T$ and set
$\Theta_T(X^T)(y,z)$ given there.
$\bb P_x$--a.s. it holds
\begin{equation}\label{secondino}
\hat Q _\epsilon(\mc R_{T,X} ) (y,z)=
 \frac{1}{ T} \int _0^T \Big\{ \sum _{ \substack{ u \in [s,s+1]:\\ \pi_T(u) \in \Theta_T(X^T)(y,z)
 }} \varphi _\epsilon (u-s)\Bigr\} \wedge \epsilon^{-1} ds\,.
\end{equation}
  For each $(y,z) \in E$ and $\epsilon>0$ we define the functions
$G_\epsilon(y,z)$ and $H_\epsilon (y,z)$ on $D(\bb R;V)$ as
\begin{align*}
& G_\epsilon (X)(y,z) := \frac{1}{T} \int_0^T \bigl| \Theta_T(X^T)(y,z) \cap  \pi_T\bigl([s+\epsilon, s+1-\epsilon]\bigr) \bigr| \wedge   \epsilon^{-1} ds \\
& H_\epsilon (X)(y,z):=\frac{1}{T} \int_0^T \bigl|
\Theta_T(X^T)(y,z) \cap \pi_T\bigl( [s+\epsilon, s+1-\epsilon]\bigr)
\bigr|  ds
 \,.
\end{align*}
By the same argument used in  identity \eqref{rio3D}, it holds
\begin{equation}\label{mercatino}
H_\epsilon (X)(y,z)=(1-2\epsilon) Q_T(X^T)(y,z)= (1-2\epsilon)\hat Q
(\mc R_{T,X} )(y,z)\,.
 \end{equation}
 Trivially, it holds $\hat Q(\mc R_{T,X})(y,z)  \geq \hat
Q_\epsilon(\mc R_{T,X})(y,z)  \geq G_\epsilon (X)(y,z) $. Using
\eqref{mercatino} and the last bounds, we can estimate
\begin{equation}\label{magmion1}
\begin{split}
 &  \bb P_x \left(\hat Q(\mc R_{T,\cdot }) (y,z) - \hat Q_\epsilon
(\mc R_{T,\cdot })(y,z)\geq \delta
\right) \\
& \qquad \leq \bb P_x \left(\hat Q (\mc R_{T,\cdot })(y,z)
-G_\epsilon(y,z) \geq \delta
\right)\\
& \qquad \leq  \bb P_x \Big(\hat Q (\mc R_{T,\cdot }) (y,z)-
H_\epsilon (y,z) \geq \delta/2 \Big)+
 \bb P_x \Big(  H_\epsilon (y,z)   - G_\epsilon (y,z) \geq \delta/2
\Big)\\
&\qquad =\bb P_x \Big(2\epsilon Q_T(X^T)(y,z) \geq \delta/2 \Big) +
 \bb P_x \Big(  H_\epsilon (y,z)   - G_\epsilon (y,z) \geq \delta/2
\Big) \,.
\end{split}
\end{equation}
In order to prove the super-exponential estimate \eqref{dz2} it is enough to prove a super-exponential estimate
for both terms in the last line of \eqref{magmion1}. 

Since, by the graphical construction, under $\bb P _x$ the process $\left\{|T Q _T (X)(y,z)
|\right\}_{T\in \mathbb R_+}$ is dominated by a Poisson process
$\left\{Z_T\right\}_{T\in \mathbb R_+}$ with parameter $r(y,z) $
we have
\begin{eqnarray*}
& &\lim_{\epsilon \downarrow 0}\varlimsup_{T\to +\infty} \frac 1T \log\left[\bb P_x \Big(2\epsilon Q_T(X^T)(y,z) \geq \delta/2
\Big)\right]\\
& & \leq\lim_{\epsilon \downarrow 0}\varlimsup_{T\to +\infty} \frac 1T \log\left[\bb P \Big(2\epsilon (Z_T+1)/T\geq \delta/2\Big)\right]\\
& &\leq \lim_{\epsilon \downarrow 0}-\Phi\left(\frac{\delta}{4\epsilon},r(y,z)\right)=-\infty\,.
\end{eqnarray*}
We used a LDP for the Poisson process (the extra $1/T$ term is irrelevant) and the explicit form of the rate functional.

 It remains to bound the last term in \eqref{magmion1}. For simplicity of notation we restrict to $T$ integer (the
 general case can be treated similarly).  We define $\psi_\epsilon (r)= r \id( r > \epsilon^{-1})$.
Given $j =0,1,\dots , T-1$ and $s \in [j,j+1)$ we have
\begin{equation*}
\begin{split}  \bigl| \Theta_T(X^T)(y,z) & \cap \pi_T\bigl([s+\epsilon,
s+1-\epsilon]\bigr) \bigr| \\ & -\bigl| \Theta_T(X^T)(y,z) \cap
\pi_T\bigl([s+\epsilon, s+1-\epsilon]\bigr) \bigr| \wedge
\epsilon^{-1}       \\ &\leq \psi_\epsilon \left( \bigl|\Theta_T(X^T)(y,z)
\cap \pi_T\bigl( [j,j+2)\bigr) \bigr|\right)\,.
\end{split}
 \end{equation*}
Hence, we can estimate
 \begin{equation}\label{ulcera}
    H_\epsilon (X)(y,z)   - G_\epsilon (X)(y,z) \leq \frac{1}{T}
 \sum_{j=0}^{T-1} \psi_\epsilon \left( \bigl|\Theta_T(X^T)(y,z) \cap \pi_T\bigl(
[j,j+2)\bigr) \bigr|\right)\,.
 \end{equation}
By the graphical construction of Markov chains, under $\bb P_x$ the
set of jump times for a jump from $y$ to $z$ can be identified with
a suitable subset of an homogeneous  Poisson point process on $\bb
R_+$ with intensity $r(y,z)$. In particular, it is possible to
define a probability measure $\mc P$ on the product space $D(\bb R_+
;V) \times D(\bb R_+;\bb N)$ such that
\begin{itemize}
\item[(i)]
 the marginal of $\mc P$ on
$D(\bb R_+;V)$ equals $\bb P_x$;

\item[(ii)]   the marginal of $\mc P$ on $D(\bb R_+;\bb N)$ is the law of a    Poisson process
 with parameter $r(x,y)$,
\item[(iii)] calling $(X_t)_{t \in \bb R_+}$ and $(Z_t)_{t \in \bb
R_+} $ the generic elements of  respectively $D(\bb R_+ ;V)$ and $
D(\bb R_+;\bb N)$, it holds $\mc P$--a.s.
$$ N_{[a,b]} (X)(y,z)  \leq Z_b-Z_a \,,\qquad \forall a<b \text{ in } \bb R_+\,.
$$
\end{itemize}
Due to the above coupling and since on the interval $[0,T]$ the
paths $X$ and $X^T$ can differ at most in $T$, we can estimate $\mc
P$--a.s.
\begin{multline}\label{ruggito}
\psi_\epsilon \left( \bigl|\Theta_T(X^T)(y,z) \cap \pi_T\bigl( [j,j+2)\bigr)
\bigr|\right)
\\\leq
\begin{cases}  \psi_\epsilon (
Z_{j+2}-Z_j)  & \text{ if } 0\leq j \leq T-2\,,\\
\psi_\epsilon ( [Z_{T}-Z_{T-1}]+Z_1+1)  & \text{ if } j= T-1\,.
\end{cases}
\end{multline}
Now we introduce the nondecreasing  function $\hat \psi_\epsilon (r):= 2 r
\id ( r> \epsilon^{-1}/2)$ satisfying the inequality   $\psi_\epsilon(a+b)
\leq \hat \psi_\epsilon (a)+\hat \psi_\epsilon(b)$. Then   \eqref{ulcera} and
\eqref{ruggito} imply $\mc P$--a.s. that
$$  H_\epsilon (X)(y,z)   - G_\epsilon (X)(y,z) \leq \frac{2}{T}
  \sum_{j=0}^{T-1} \hat \psi_\epsilon (
Z_{j+1}-Z_j+1 )\,.
 $$
At this point we recall that under $\mc P$  the random variables
$\Big(Z_{j+1}-Z_j\Big) _{0\leq j \leq T-1}$ are independent Poisson
random variables with  parameter $r(y,z)$. Hence we can estimate
\begin{eqnarray}
& &\lim_{\epsilon \downarrow 0}\varlimsup_{T\to +\infty}\frac 1T\log\left[\bb P_x \Big(  H_\epsilon (y,z)    - G_\epsilon (y,z)
\geq \delta/2 \Big)\right] \nonumber \\
& &=\lim_{\epsilon \downarrow 0}\varlimsup_{T\to +\infty}\frac 1T\log\left[ \mc P  \Big(  H_\epsilon (y,z)   - G_\epsilon
(y,z) \geq \delta/2 \Big)\right]\nonumber\\
& & \leq \lim_{\epsilon \downarrow 0}\varlimsup_{T\to +\infty}\frac 1T\log\left[\mc P \Big (\frac{2}{T}
  \sum_{j=0}^{T-1} \hat \psi_\epsilon (
Z_{j+1}-Z_j+1 )\geq \delta /2\Big)\right]\nonumber\\
& &\leq \lim_{\epsilon \downarrow 0}-I_\epsilon (\delta/2)=-\infty\,.
\label{federer-tsonga}
\end{eqnarray}
In the above chain of inequalities we used Cramer Theorem for the sum of the independent
random variables $2\hat \psi_\epsilon(Z_{j+1}-Z_j+1)$ calling $I_\epsilon$ the associated
rate function. 
 The divergence in the last line follows by the following argument. Let $\Lambda_\epsilon(\lambda):=\log \mathbb E\left(e^{\lambda2\hat \psi_\epsilon(Z_{1}-Z_0)}\right)$.
By the Monotone Convergence Theorem
$\Lambda_\epsilon(\lambda) $ converges to zero for each $\lambda  \in \bb R$ as $\epsilon$ goes to zero.
Since  the rate function $I_\epsilon$ is the Legendre transform of $\Lambda_\epsilon$, we get for each fixed $\lambda \in \bb R$  that  $$I_\epsilon(\delta/2) \geq  \frac{\delta \lambda}{2}-\Lambda_\epsilon(\lambda)\,.$$ Hence, $\liminf_{  \epsilon \downarrow 0 } I_\epsilon  ( \delta /2) \geq  \delta \lambda /2$. By the arbitrariness of $\lambda $  we get the thesis.

\section{Exponential approximations: Proof of Proposition \ref{cervo3}}\label{dim_cervo3}

The measurability of $\bar Q$ can be checked by straightforward
arguments. Let us prove that $\bar Q_\epsilon$ is continuous w.r.t.
the bounded weak* topology of $L_+^1(E)$.
 As stated  in Prop. \ref{cervo2} each map $\hat Q_\epsilon(y,z) :
\mc M_S\to [0,\epsilon ^{-1}]$ is continuous and bounded.
In addition it holds $\|\bar Q_\epsilon (R)\|\leq \epsilon ^{-2}$
for all $R\in \mc M_S$. The thesis then follows from Corollary 2.7.3
in \cite{Me}.

\subsection{Proof of \eqref{dz2bis}}
Due to Proposition \ref{cervo1pezzo} the law of $\hat Q $ under
$\Gamma_{T,x}$ is the same of the law of $Q_T(X^T)$ under $\bb P_x$.
Moreover, it holds $Q_T(X^T) \in L^1 _+(E)$ $\bb P_x$--a.s. In
particular, we get that $\hat Q= \bar Q$ $\Gamma_{T,x}$--a.s. In
addition, by Proposition \ref{t:etem}, we have
\begin{equation}\label{milano}
\lim _{\ell \uparrow  +\infty}   \varlimsup_{T \uparrow +\infty}
\frac{1}{T} \log  \Gamma_{T,x} \left( \| \hat  Q \|\geq \ell \right)
=-\infty\,.
\end{equation}

Due to \eqref{milano} in order to prove \eqref{dz2bis} we only need
to show for any $\ell
>0$ that
\begin{equation}\label{roma}
\lim_{\epsilon  \downarrow 0} \varlimsup_{T \uparrow \infty}
\frac{1}{T} \log  \Gamma_{T,x} \left(  d_\mathfrak{a} ( \bar  Q  ,
\bar Q_\epsilon) > \delta\,,\, \|\hat Q\|\leq \ell  \right)
=-\infty\,.
\end{equation}
 Since $\mathfrak{a}\in \mc A$, there exits $\bar n \geq 1$ such
that $\|a_n\|_\infty \leq \delta /(2 \ell)$ for all $n \geq \bar n$.
Note that, since $\hat Q (y,z)(R) \geq \bar  Q_\epsilon (y,z)(R)$,
it holds $\|\hat Q(R)\| \geq \|\bar Q_\epsilon(R)\|$  and $\|\hat
Q(R)\| \geq \|\hat Q(R)-\bar Q_\epsilon(R)\|$ for any $R \in \mc
M_S$. Then for any $n \geq \bar n$ we have $|< \hat Q(R)-\bar
Q_\epsilon(R) , a_n>| \leq \delta /2$ if $\|\hat Q(R)\|\leq \ell$.
Therefore, in order to prove \eqref{roma} we only need to show for
any $\ell >0$ that
\begin{equation}\label{romaz}
\lim_{\epsilon  \downarrow 0} \varlimsup_{T \uparrow \infty}
\frac{1}{T} \log  \Gamma_{T,x} \left(\exists n: 1\leq n\leq \bar n
\text{ s.t. } |< \hat Q-\bar Q_\epsilon, a_n>| > \delta /2  \,,\,
\|\hat Q\|\leq \ell \right) =-\infty
\end{equation}
Since $a_n \in C_0(E)$ we can find a finite subset $E'\subset E$
such that $|a_n (e)|\leq \delta /4\ell$ for all $n:1\leq n \leq \bar
n$ and $e \in E \setminus E'$. Estimating $$ |< \hat Q-\bar
Q_\epsilon, a_n>| \leq \sum_{(y,z) \in E'} \bigl| \bigl(\hat Q(y,z)-
\bar Q_\epsilon(y,z)\bigr) a_n(y,z)\bigr|+ \|\hat Q-\bar
Q_\epsilon\| \sup _{ e \in E \setminus E'} |a_n(e)|\,,$$ we reduce
the proof of \eqref{romaz} to the proof of
\begin{equation}
\lim_{\epsilon  \downarrow 0} \varlimsup_{T \uparrow \infty}
\frac{1}{T} \log  \Gamma_{T,x} \left( | \hat Q(y,z)-\bar
Q_\epsilon(y,z) |
> \beta \right ) =-\infty\,, \qquad \forall (y,z) \in E, \; \forall \beta
>0\,.
\end{equation}
This follows from \eqref{dz2}.

\subsection{Proof of \eqref{dz1bis}}  By arguments similar to the ones used in the previous proof
the thesis follows thanks   to the bound \eqref{armadio} in Lemma
\ref{schnell}  and \eqref{dz1}.

\section{Birth and death processes}\label{s:BD}

Birth and death processes are  nearest--neighbor continuous time
Markov chains on $\bb Z_+$ with jump rates $r(k,k+1)=b_k$ and
$r(k+1,k)=d_{k+1}$, $k\ge 0$. We assume the birth rate $b_k$ and the
death rate $d_k$ to be strictly positive. We also assume
\begin{equation}\label{normale}
  Z:= \sum_{k=0}^{+\infty}
 \frac{b_0 b_1\cdots b_{k-1}}{d_1 d_2 \cdots d_{k} }<+\infty
\end{equation}
and \begin{equation}\label{urca} \sum_{k=0}^{+\infty} \frac{d_1d_2
\cdots d_k}{ b_1b_2\cdots b_k}= +\infty\,.
\end{equation}
Then  assumptions (A1)--(A4) holds. Indeed, (A1) and (A3) are
trivially satisfied.  Due to the presence of a leftmost point (the
origin), equation \eqref{invariante} reduces to the detailed balance
equation and admits  normalizable  solutions if and only if
\eqref{normale} is fulfilled. In particular, one obtains a
  unique invariant probability given by
\begin{equation}
  \label{pibd}
 \pi(0)=\frac 1Z\,,\qquad  \pi(k) = \frac 1Z \,\frac{b_0 b_1\cdots b_{k-1}}{d_1 d_2 \cdots d_{k} }
  \qquad k\geq 1\,.
\end{equation}
 Having \eqref{normale}, condition \eqref{urca} is equivalent
to non--explosion (A2) (combine Corollary 3.18 in \cite{C0} with
\eqref{urca}) and can be rewritten as $\sum_{k=1}^\infty 1/
(\pi(k)b_{k})=+\infty$. Note that condition \eqref{urca} is
equivalent to recurrence (combine \cite[Ex.\ 1.3.4]{N} with
\cite[Th. 3.4.1]{N}.
 Under the above assumptions, the logarithmic Sobolev inequality
holds if and only if (see Table 1.4 in \cite[Ch. 1]{C})
\begin{equation}\label{lollo}
\sup_{k\geq 1} \pi ([k,+\infty)) \,
\log \left( \frac{1}{\pi ([k,+\infty)) }\right)  \sum
_{j=0}^{k-1} \frac{1}{\pi(j) b_j} <+\infty\,.
\end{equation}

\smallskip

\emph{Possible absence of exponential tightness of the empirical
measure}. We first discuss a case in which the empirical measure
fails to be exponentially tight. Consider constant birth and death
rates, i.e.\ $b_k=\beta$ and $d_k=\delta$. Then \eqref{normale}  and
\eqref{urca} together are  equivalent to the condition
$\gamma:=\beta/\delta\in (0,1)$. In particular, $\pi$ is geometric
with parameter $\gamma$, i.e. $\pi(k)=(1-\gamma) \gamma^k$. Consider
an event in which in the time interval $[0,T]$ there are  $O(T)$
jumps (typical behavior) but all the jumps are to the right
(atypical behavior). The probability of such an event is ``only''
exponentially small in $T$ and therefore the empirical measure
cannot be exponentially tight. To be more precise, we write $N_T$
for the number of jumps performed in the time interval $[0,T]$.
Since the holding time at site $k$ is exponential of parameter
$\beta$ if $k=0$ and $\beta +\delta$ if $k \geq 1$, $N_T$
stochastically dominates [is stochastically dominated by] a Poisson
random variable with mean $\beta T$ [$(\beta+\delta)T$]. Hence, with
probability $1-o(1)$, $N_T$ has value in $I:=[\beta T/2,
2(\beta+\delta)T]$. By conditioning on $N_T$, it is then simple to
check that with probability at least $(1-o(1))[ \beta /
(\beta+\delta)]^{2(\beta+\delta)T-1} $ the following event $\mc A_T
$ takes place: the random variable $N_T$ has value in $I$ and all
the jumps are to the right. The event $\mc A_T$ implies  $\mu_T=
\sum_{i=0} ^{N_T} \delta_i /T$. Take now a compact set $\mc K\subset
\mc P(V)$. By Prohorov's theorem, $\mc K$ is a tight family of
probability measures and therefore, given $\epsilon
>0$, there exists a compact (finite) set $K\subset V$ such that
$\mu(K^c)\leq\epsilon$ for all $\mu \in \mc K$. Taking $T$ large
enough, under the event $\mc A_T$ the empirical measure $\mu_T$
cannot fulfills the above requirement. Hence
$$ \bb P_0( \mu_T \not \in \mc K) \geq \bb P_0 ( \mu_T(K^c)>
\epsilon)\geq \bb P_0 (\mc A_T) \geq  (1-o(1))[ \beta /
(\beta+\delta)]^{2(\beta+\delta)T-1}\,.$$ This estimate proves
that the empirical measure cannot be exponentially tight.  In
particular neither Condition~\ref{t:ccomp}  nor
\ref{t:ccompls} holds (even with $\sigma=0$).

\smallskip

\emph{Condition \ref{t:ccomp}}.
 Assume now
\begin{equation}
  \label{cdvbd}
  \lim_{k\to\infty} d_k =+\infty,
  \qquad
  \varlimsup_{k\to\infty} \: \frac{b_k}{d_k} <1.
\end{equation}
Trivially, \eqref{normale} and \eqref{urca} are satisfied. We show
that Condition~\ref{t:ccomp}  holds. As $u_n$ we
pick the constant sequence $u(k)=A^k$, $k\in\bb Z_+$ for some $A>1$
to be chosen later. Since $u_n$ does not depend on $n$, it is enough
to check Condition \ref{t:ccomp}. Items (i)--(iv) then hold
trivially; moreover setting $d_0:=0$ we get
\begin{equation*}
  v(k) = -\frac{Lu }{u} \,(k) =
d_k\Big(1-\frac{1}{A}\Big)+b_k(1-A)\,,
  \qquad k\in\bb Z_+.
\end{equation*}
Since $r(k) = b_k+d_k$, for each $\sigma\in (0,1)$ we can write
$v(k) = \sigma r(k)+d_k(1-\sigma-1/A)-b_k(A-1+\sigma)$.
 By \eqref{cdvbd},  choosing  $A$ large   items (v) and (vi) hold.
Observe that \eqref{cdvbd} is satisfied when $d_k=k$ and
$b_k=\lambda\in (0,+\infty)$. In this case $\pi$ is Poisson with
parameter $\lambda$. This implies that $e^{-\lambda}\lambda^k/k!
\leq \pi([k,+\infty)) \leq \lambda^k /k!$ (for the last bound
estimate $\pi(i)\leq e^{-\lambda} \lambda^i/(k-i)!$ for $i \geq k$).
Using these bounds, by simple computations one can check from
\eqref{lollo} that
 the logarithmic
Sobolev inequality \eqref{ls} does not hold. This shows there are
cases in which Condition~\ref{t:ccomp} holds but
Condition~\ref{t:ccompls} does not.

\smallskip

\emph{Condition \ref{t:ccompls}}. Let now focus our attention on
Condition \ref{t:ccompls}. As already mentioned, the validity of the
logarithmic Sobolev inequality is equivalent to \eqref{lollo}
(assuming \eqref{normale} and \eqref{urca}).


We next exhibit a choice in which Condition~\ref{t:ccompls} holds.
We take $b_k= (k+1)$ and $d_{k+1}= 2b_k$ for $k \geq 0$. Observe
that such rates satisfy \eqref{cdvbd}, and therefore \eqref{normale}
and \eqref{urca}. The invariant probability $\pi$ is $\pi(k)=
2^{-k-1}$. In remains to estimate  $ \sum_{j=0}^{k-1} (\pi(j) b_j)^{-1}= \sum_{j=1}^k 2^j/j$.
Supposing for simplicity $k$ even, we observe that  $\sum_{j=1}^{k/2} 2^j/j \leq (k/2) 2^{k/2}$
while $\sum_{j=k/2}^{k} 2^j/j \leq (2/k) \sum_{j=k/2}^{k} 2^j =(2/k) 2^{k/2} \sum _{j=0}^{k/2-1} 2^j \\=
(2/k) 2^{k/2}(2^{k/2}-1) $. Hence $ \sum_{j=0}^{k-1} (\pi(j) b_j)^{-1}\leq C k 2^{k/2}+ C 2^k /k$.
From these bounds it is immediate to get \eqref{lollo}. In addition,
since  $r(k)\sim k$ we deduce immediately that also item
(ii) in Condition~\ref{t:ccompls} holds, thus completing the check of
 Condition~\ref{t:ccompls}.


\emph{Violation of the  LDP in the strong topology of $L^1_+(E)$}.
By exhibiting a concrete example, we   show that - under
Condition~\ref{t:ccomp}  -
Theorem~\ref{LDP:misura+flusso} does not hold in the strong topology
of $L^1_+(E)$. We choose the birth and death rates as $b_k=(k+1)/2$
and $d_k =k$; in particular $\pi$ is geometric with parameter $1/2$.
Since \eqref{cdvbd} holds, Condition~\ref{t:ccomp} is satisfied. We shall show that the level sets of $I$ in
\eqref{rfq} are not compact in the strong topology of $L^1_+(E)$.
Set
\begin{equation*}
  \begin{split}
   \mu^n & := \big( 1-\tfrac 1n\big) \, \pi
             + \tfrac 1{2n} \big[ \delta_n +\delta_{n+1}\big]
    \\
    Q^n  & :=  \big( 1-\tfrac 1n\big) \, Q^\pi
             + \tfrac 12 \big[ \delta_{(n,n+1)} +\delta_{(n+1,n)}\big].
  \end{split}
\end{equation*}
While $\{\mu^n\}$ converges to $\pi$ in $\mc P(\bb Z_+)$, observe
that $\{Q^n\}$ converges to $Q^\pi$ in the bounded weak* topology of
$L^1_+(E)$ but it is not compact in the strong topology of
$L^1_+(E)$. Since $\div Q^n=0$, it is simple to check that
$\varlimsup_n I(\mu^n,Q^n) < +\infty$. This implies that the level
sets of $I$ are not compact in the strong topology of $L^1_+(E)$.

\appendix

\section{Proof of \eqref{vci}}\label{iobimbo}
 We call $\bar I(\mu,Q)$ the  r.h.s. of \eqref{vci}.  Trivially it
holds  $\bar I(\mu,Q)=+\infty=I(\mu,Q)$ if $\div Q\not =0$. In the
sequel we assume $\div Q=0$. Then,  equation \eqref{vci} reads
$I(\mu,Q)= \sup_{F \in C_c (E)} I_F(\mu,Q)$ where
$I_F(\mu,Q):=\langle Q,F \rangle -\langle \mu, r^F- r \rangle$. If
for some $y \in V$ and $(y,z) \in E$  it holds $\mu(y)=0$ and
$Q(y,z)>0$,  then taking $F= \lambda \delta _{(y,z)}$ with $\lambda
\to +\infty$ we obtain that $\bar I(\mu,Q)=\infty$. On the other
hand
$$I(\mu,Q) \geq \Phi( Q(y,z), Q^\mu(y,z) )= \Phi( Q(y,z), 0)
=+\infty\,.$$ As a consequence, from now on we can restrict to
$(\mu,Q)$ such that $\div Q=0$ and  $Q(y,z)=0$ for all $(y,z)\in E$
with $\mu(y)=0$.
Calling $E_+:=\{ (y,z) \in E\,:\, \mu(y)>0  \}$ we get that
$$ I_F (\mu,Q)= \sum _{(y,z) \in E_+} \left\{ Q(y,z) F(y,z)- \mu(y)
r(y,z) ( e^{F(y,z)}-1) \right\}\,.$$ At this point, it is simple to
check  that, varying $F(y,z)$,  the supremum of the above  addendum
is given by $\Phi ( Q(y,z), Q^\mu(y,z) )$ and the value of the above
addendum for $F(y,z)=0$ is zero. Hence,
$$\bar I(\mu,Q)= \sum _{(y,z)\in E_+}\Phi ( Q(y,z), Q^\mu(y,z)
)=\sum _{(y,z)\in E}\Phi ( Q(y,z), Q^\mu(y,z) )\,.$$ We now claim
that  the  above expression is $+\infty$ if $\langle \mu,
r\rangle=+\infty$, thus concluding the proof.  To this aim we
observe that for $0\leq q < p/2$ it holds $ \Phi (q,p)\geq p(1-\log
2) /2$. Indeed, the thesis is trivially true if $q=0$, while for
$q>0$ we can write $\Phi(q,p)=p f(q/p)$ where $f(x)=x \log x +1-x$.
Since $f(x) $ is decreasing for  $0<x<1$, one has $\Phi(q,p) \geq p
f (1/2)$ for  $0\leq q < p/2$.
 Hence, setting $c:= 2/(1-\log 2) $, our claim follows from the
bound
\begin{equation*}
\begin{split}\langle
\mu, r \rangle& = \sum _{(y,z) \in E} Q^\mu (y,z)  \\
& \leq
 \sum _{
\substack{
(y,z)\in E\,:\, \\
Q(y,z)< Q^\mu (y,z)/2 } }c\ \Phi( Q(y,z), Q^\mu(y,z) ) + \sum
_{\substack{ (y,z) \in E \,:\, \\ Q(y,z)\geq  Q^\mu (y,z)/2 }}2
Q(y,z) \\& \leq\sum _{(y,z)\in E}c\ \Phi ( Q(y,z), Q^\mu(y,z) )
+2\|Q\|_1\,.
\end{split}
\end{equation*}

\section{An example with discontinuous divergence} \label{div_disco}
Consider the oriented  graph $(V,E)$ where $V= \bb N \cup \{v,w\} $
and $E$ is given by the oriented bonds of the form $(v,n)$,
$(n,w)$,$(w,v)$ for some $n \in \bb N$. For each $n \in \bb N$ we
define $Q^{(n)}$ as the flow of unitary flux associated to the cycle
$(v,n,w,v)$, i.e.\ $Q^{(n)}= \id_{(v,n)} + \id _{(n,w)}
+\id_{(w,v)}$. We claim that $Q^{(n)}$ converges to $Q:= \id
_{(w,v)}$ in  $L^1_+(E)$ (endowed of the bounded weak* topology).
Since $\|Q^{(n)}\|=3$, the sequence $\big(Q^{(n)}\big)_{n \in \bb
N}$ is bounded in the strong topology of $L^1_+(E)$. In particular,
$Q^{(n)} \to Q$ in the bounded weak* topology if and only if
$Q^{(n)} \to Q$ in the  weak* topology, and therefore  if and only
if $\langle \phi, Q^{(n)}\rangle \to \langle \phi, Q\rangle $ for
each $\phi \in C_0(E)$. By construction we have
$$\langle \phi, Q^{(n)}\rangle = \phi(v,n)+\phi(n,w)+ \phi (w,v) \to
\phi(w,v) = \langle \phi, Q\rangle \,,
$$
thus concluding the proof of our claim.

We observe that, despite $\div Q^{(n)}=0$ for all $n \in \bb N $, it
holds $\div Q \not =0$. This example shows that the map $L^1_+(E)
\ni Q \to \div Q(x)\in \bb R$, with $x \in V$, is not in general a
continuous map.

\subsection*{Acknowledgements}

We thank the   referees for the careful reading and the useful suggestions.

\end{document}